\documentclass[a4paper,fleqn]{cas-sc}

\usepackage[ruled,vlined,linesnumbered]{algorithm2e}

\usepackage[authoryear,longnamesfirst,sort&compress]{natbib}
\usepackage{pgfplots}
\pgfplotsset{compat=1.16}

\usepackage{graphicx}
\usepackage{hyperref}

\usepackage{amsmath}
\usepackage{xspace}
\usepackage{amssymb}
\usepackage{amsthm,subcaption}

\usepackage{tikz,float}
\usetikzlibrary{positioning, arrows.meta, calc}
\usetikzlibrary{arrows.meta, decorations.markings, calc}

\def\tsc#1{\csdef{#1}{\textsc{\lowercase{#1}}\xspace}}
\tsc{WGM}
\tsc{QE}

\begin{document}
\let\WriteBookmarks\relax
\def\floatpagepagefraction{1}
\def\textpagefraction{.001}

\shorttitle{Moment matching quadratic manifolds}

\shortauthors{R. Padhi, S. Gugercin}  

\title[mode = title]{Beyond linear subspaces: Nonlinear moment matching meets quadratic manifolds}

\author[1]{Reetish Padhi}[orcid=0009-0000-9165-9844]

\cormark[1]

\ead{reetishp@vt.edu}

\ead[url]{https://rewtus.github.io/reetish/}

\credit{Conceptualization, Investigation, Methodology, Software, Formal analysis, Validation, Writing – original draft, Writing – review and editing}

\affiliation[1]{organization={Department of Mathematics, Virginia Tech}, 
            city={Blacksburg},
            postcode={24061}, 
            state={Virginia},
            country={USA}}

\author[2]{Serkan Gugercin}[orcid = 0000-0003-4564-5999]

\ead{gugercin@vt.edu}

\ead[url]{https://gugercin.math.vt.edu}

\credit{Conceptualization, Investigation, Methodology, Supervision, Formal analysis, Validation, Writing – review and editing}

\affiliation[2]{organization={Department of Mathematics, Virginia Tech},
            city={Blacksburg},
            postcode={24061}, 
            state={Virginia},
            country={USA}}
            
\cortext[1]{Corresponding author}

\begin{abstract}
Quadratic manifold-based model order reduction offers a viable pathway to circumvent the limitations of linear subspaces for linear control systems characterized by slow Kolmogorov $n$-width decay. However, a system-theoretic framework for constructing such quadratic approximations remains absent from the literature. This paper presents a system-agnostic, optimization-free framework for the direct construction of quadratic projection matrices. We prove that the synthesized reduced-order model  matches the nonlinear moments of the full-order system and preserves its exact center manifold mapping, thereby ensuring asymptotic tracking of steady-state outputs under specific input classes. Numerical results on transport-dominated benchmark problems, namely, the one-dimensional damped wave and  advection equations, show that the proposed framework achieves high-fidelity trajectory reconstruction within a significantly reduced-dimensional state space, yielding substantial online computational savings.
\end{abstract}

\begin{keywords}
 Quadratic manifolds\sep Model order reduction \sep Nonlinear moment matching \sep Kolmogorov $n$-width \sep LTI systems \sep Transport dominated problems
\end{keywords}

\maketitle

\theoremstyle{plain}  
\newtheorem{thm}{Theorem}[section]
\newtheorem{lem}[thm]{Lemma}

\newtheorem{cor}[thm]{Corollary}
\newtheorem{alg}[thm]{Algorithm}
\theoremstyle{definition} 
\newtheorem{defi}[thm]{Definition}
\newtheorem{conj}[thm]{Conjecture}
\newtheorem{claim}[thm]{Claim}
\newtheorem{prop}[thm]{Proposition}
\newtheorem{notation}[thm]{Notation}

\theoremstyle{remark}  
 \newtheorem{rem}[thm]{Remark}
\newtheorem{com}[thm]{Comment}

\newcommand*{\reetish}{\textcolor{red}}
\newcommand*{\serkan}{\textcolor{blue}}

\newcommand{\cA}{{\mathcal A}}
\newcommand{\cB}{{\mathcal B}}
\newcommand{\cC}{{\mathcal C}}
\newcommand{\cD}{{\mathcal D}}
\newcommand{\cE}{{\mathcal E}}
\newcommand{\cF}{{\mathcal F}}
\newcommand{\cG}{{\mathcal G}}
\newcommand{\cH}{{\mathcal H}}
\newcommand{\cI}{{\mathcal I}}
\newcommand{\cJ}{{\mathcal J}}
\newcommand{\cK}{{\mathcal K}}
\newcommand{\cL}{{\mathcal L}}
\newcommand{\cM}{{\mathcal M}}
\newcommand{\cN}{{\mathcal N}}
\newcommand{\cO}{{\mathcal O}}
\newcommand{\cP}{{\mathcal P}}
\newcommand{\cQ}{{\mathcal Q}}
\newcommand{\cR}{{\mathcal R}}
\newcommand{\cS}{{\mathcal S}}
\newcommand{\cT}{{\mathcal T}}
\newcommand{\cU}{{\mathcal U}}
\newcommand{\cV}{{\mathcal V}}
\newcommand{\cW}{{\mathcal W}}
\newcommand{\cX}{{\mathcal X}}
\newcommand{\cY}{{\mathcal Y}}
\newcommand{\cZ}{{\mathcal Z}}

\newcommand{\bA}{{\mathbf A}}
\newcommand{\bB}{{\mathbf B}}
\newcommand{\bC}{{\mathbf C}}
\newcommand{\bD}{{\mathbf D}}
\newcommand{\bE}{{\mathbf E}}
\newcommand{\bF}{{\mathbf F}}
\newcommand{\bG}{{\mathbf G}}
\newcommand{\bH}{{\mathbf H}}
\newcommand{\bI}{{\mathbf I}}
\newcommand{\bJ}{{\mathbf J}}
\newcommand{\bK}{{\mathbf K}}
\newcommand{\bL}{{\mathbf L}}
\newcommand{\bM}{{\mathbf M}}
\newcommand{\bN}{{\mathbf N}}
\newcommand{\bO}{{\mathbf O}}
\newcommand{\bP}{{\mathbf P}}
\newcommand{\bQ}{{\mathbf Q}}
\newcommand{\bR}{{\mathbf R}}
\newcommand{\bS}{{\mathbf S}}
\newcommand{\bT}{{\mathbf T}}
\newcommand{\bU}{{\mathbf U}}
\newcommand{\bV}{{\mathbf V}}
\newcommand{\bW}{{\mathbf W}}
\newcommand{\bX}{{\mathbf X}}
\newcommand{\bY}{{\mathbf Y}}
\newcommand{\bZ}{{\mathbf Z}}

\newcommand{\system}{{ \mathbf{\Sigma}}}
\newcommand{\bx}{\mathbf{x}}
\newcommand{\by}{\mathbf{y}}
\newcommand{\bu}{\mathbf{u}}
\newcommand{\bw}{\mathbf{w}}
\newcommand{\bh}{\mathbf{h}}
\newcommand{\bz}{\mathbf{z}}

\newcommand{\tV}{\widetilde V}

\newcommand{\w}{\ensuremath{\omega}}
\newcommand{\hw}{\ensuremath{\hat\omega}}

\newcommand{\hlin}{\bh_1}
\newcommand{\htwo}{\bh_2}
\newcommand{\hthree}{\bh_3}
\newcommand{\hfourone}{\bh_4}

\newcommand{\setV}{\operatorname{Im(V)}}
\newcommand{\setpi}{\Gamma}

\newcommand{\p}{\mathbf \Pi}
\newcommand{\prtwo}{\ensuremath{\Pi_r^{(2)}}}
\newcommand{\ptwo}{\ensuremath{\Pi^{(2)}}}

\newcommand{\bp}{\mathbf{p}}
\newcommand{\pr}{\boldsymbol{\pi_r}}

\newcommand{\matlab}{\textsc{MATLAB}\xspace}

\section{Introduction}
Model Order Reduction (MOR) has emerged as a critical tool for the simulation and control of high-fidelity, large-scale dynamical systems. By approximating the state of a high-dimensional system within a low-dimensional subspace/manifold, MOR enables significant computational savings while maintaining key mathematical properties of the original system. Traditional reduction techniques focus on constructing a reduced model by projecting the state onto a linear subspace. For linear time-invariant (LTI) systems, (linear) projection-based methods such as interpolatory methods \cite{morAntBG20}, balancing-based approaches \cite{BenB17}, and Proper Orthogonal Decomposition (POD) \cite{volkwein2011pod}  are well-established. However, the efficacy of linear MOR is fundamentally governed by the Kolmogorov $n$-width \cite{pinkus2012n}, which represents the worst-case error arising from the projection of the solution manifold onto the best-possible linear subspace of dimension $r \ll n$. For many elliptic or parabolic PDEs, the Kolmogorov $n$-width decays exponentially with $r$, allowing for low dimensional reduced order models (ROM). Conversely, for hyperbolic or transport-dominated problems, the decay of these widths is significantly slower  \cite{kol_wave_decay, peherstorfer2022breaking}. This ``slow decay" constitutes a barrier to reducibility for linear methods. In \cite{unger2019kolmogorov}, the authors show that for LTI systems, the Kolmogorov $n$-widths coincide with the Hankel singular values \cite{ACA05}, thus connecting the concept of the Kolmogorov $n$-widths to system-theoretic objects.

Nonlinear dimensionality reduction techniques have emerged as a means to circumvent the Kolmogorov $n$-width barrier using the theory of nonlinear manifolds \cite{QM_framework}. Among these, machine learning approaches like the autoencoder-based frameworks, utilize deep learning architectures to map high-dimensional data to a latent space via a nonlinear mapping \cite{buchfink2023symplectic, otto2023learning,fresca2021comprehensive,fresca2022pod,kim2022fast,kadeethum2022non}. Alternatively, dictionary and localized basis methods, e.g., \cite{amsallem2012nonlinear, daniel2022physics, geelen2022localized, rewienski2003trajectory}, construct a library of localized, typically time-invariant bases that partition the state space, adaptively selecting the optimal local subspace online during system evolution. 

Shift based methods consider explicit, time-dependent spatial shifts or transport maps to align moving features along solution trajectories~\cite{reiss2018shifted, black2020projection, reiss2021optimization, papapicco2022neural, burela2023parametric, krah2025robust}.
Freezing techniques and symmetry-based reduction frameworks exploit symmetry/equivariance of the original system. They involve Lie group actions and algebraic conditions to transform the governing equations into a moving frame of reference, effectively ``freezing'' traveling fronts into stationary profiles before performing dimensionality reduction~\cite{rowley2000reconstruction, rowley2003reduction, beyn2004freezing, ohlberger2013nonlinear}.
A comprehensive review and classification of these nonlinear approaches, particularly for transport-dominated problems, can be found in~\cite{hesthaven2026nonlinear}.

Existing literature on quadratic manifolds \cite{BARNETT_QM1,benner2023quadratic,geelen2023operator,sharma2023symplectic,schwerdtner2024greedy,schwerdtner2025empirical,schwerdtner2025online,paxton2026fast,glas2026structure,rutzmoser2017generalization,schwerdtner2024online} has considered different methods of picking the linear and quadratic projection matrices. For instance, \cite{ rutzmoser2017generalization} utilizes structural properties of the governing equations, while \cite{geelen2023operator} proposes an Operator Inference (OpInf) approach, where the projection subspaces are learned in a data-driven fashion by solving an optimization problem over state snapshots. In \cite{benner2023quadratic}, the authors propose a quadratic decoder approach for the approximation of nonlinear systems. A greedy approach of selecting quadratic projection bases has been presented in \cite{schwerdtner2024greedy} which was adapted to an online streaming setting in \cite{schwerdtner2024online}. On the other hand, in \cite{paxton2026fast}, the quadratic projection matrices are constructed by solving an optimization problem over the Stiefel manifold. Structure preserving model reduction using quadratic manifolds for Port-Hamiltonian systems has been studied in \cite{sharma2023symplectic,glas2026structure}. In \cite{QM_framework}, the authors present a unifying differential-geometric framework for nonlinear manifold model reduction. Higher order polynomial and rational manifold approximations have been explored in \cite{geelen2023learning,geelen2024learning,klein2025entropy,buchfink2024approximation}.

While quadratic manifold approaches offer a powerful alternative to linear subspaces, they yield reduced systems with higher-order state (nonlinear) dependencies. This structural complexity often makes it difficult to establish direct theoretical comparisons or rigorous statements relating the full- and reduced-order models. Furthermore, existing methods typically rely on optimization formulations to construct these manifolds, which require empirical hyperparameter tuning. To overcome these limitations, the main goal of this work is to provide a purely system-theoretic construction of the quadratic manifold.

A natural choice for establishing system-theoretic guarantees is interpolatory (also known as moment-matching based) model reduction, which traditionally focuses on rational interpolation of transfer functions in the frequency domain; see \cite{morAntBG20} for a comprehensive survey on linear time-invariant systems. For structured nonlinear systems, such as bilinear and quadratic-bilinear systems, the input-output behaviour is characterized by Volterra kernels \cite{rugh1981nonlinear} and the interpolatory methods (moment matching) extends to interpolating these kernels or multivariate transfer functions; see, e.g., \cite{flagg2015multipoint, benner2024structured, werner2021structure,gosea2018data,benner2015two, breiten2010krylov, gu2011qlmor, breiten2012interpolation,  bai2006projection} and the references therein. Alternatively, in \cite{astolfi2010}, the notion of moment matching was reformulated as matching the steady-state output response of a dynamical system driven by an autonomous signal generator. In this framework, ``nonlinear moments'' of a dynamical system is characterized by an invariant center manifold mapping satisfying a generalized partial differential equation (invariance equation). This formulation has been expanded to quadratic-bilinear and general nonlinear systems \cite{bai2022model, scarciotti2017nonlinear,simard2024parameterization}, and extended to non-intrusive, data-driven settings \cite{scarciotti2017data, moreschini2025moment}. We refer the reader to \cite{astolfi_10year_survey, scarciotti2024interconnection} for a detailed survey of these methods.

While center-manifold-based moment matching provides an elegant framework for analyzing steady-state output responses, it does not directly yield explicit projection bases for state-space model reduction. In that setting, the invariant manifold mapping remains an abstract time-domain evaluation operator rather than an explicit tool for constructing trial spaces. In this paper, we bridge the gap between nonlinear moment-matching theory and quadratic manifold MOR by reformulating moment matching within a quadratic projection framework. Specifically, Theorem~\ref{thm:quad_pi} establishes that a quadratic manifold structure emerges naturally for a special class of signal generators, directly yielding explicit projection matrices. This formulation provides a practical, computationally efficient approach to construct quadratic reduced-order models with system-theoretic guarantees. By shifting the paradigm from empirical snapshot optimization to system-theoretic interpolation, we prove that the quadratic reduced system matches the nonlinear moments generated by a user-specified signal space (Figure~\ref{fig:introduction_schematic}). The primary contributions of this work are summarized below:

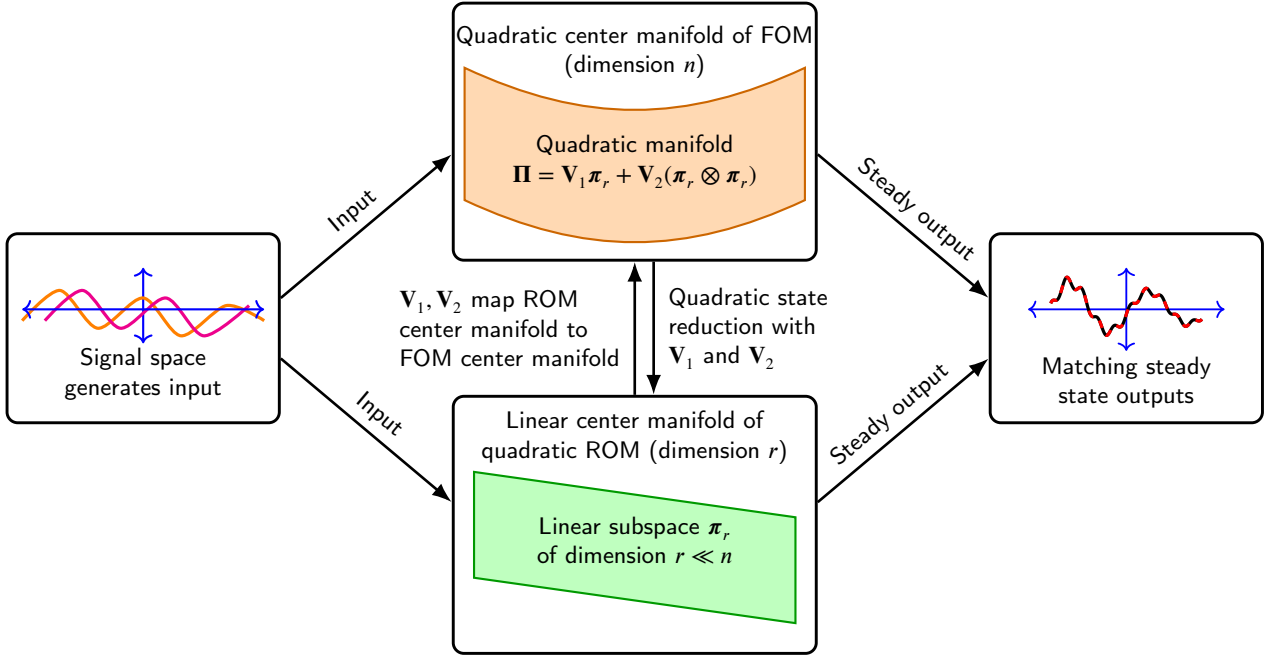
\begin{figure}
    \centering
    \def\r{2.5}

\begin{tikzpicture}[
    block/.style={
        draw, 
        rectangle, 
        line width=1.1pt, 
        rounded corners=4pt, 
        align=center,
        fill=white
    },
    image_placeholder/.style={
        draw=blue, 
        dashed, 
        line width=0.8pt, 
        fill=blue!5, 
        text=blue, 
        font=\small\ttfamily,
        minimum width=2.4cm,
        minimum height=0.8cm,
        inner sep=4pt,
        rounded corners=2pt
    },
    arrow/.style={
        -Latex, 
        thick, 
        line width=1pt
    }
]

    \node[block, minimum width=3.6cm, minimum height=2.5cm] (gen) at (-6.5, 0) {};

    \node[block, minimum width=4.8cm, minimum height=3.4cm] (sys1) at (0, 2.6) {};

    \node[block, minimum width=4.8cm, minimum height=3.4cm] (sys2) at (0, -2.6) {};

    \node[block, minimum width=3.6cm, minimum height=2.5cm] (comp) at (6.5, 0) {};

    \draw[line width=0.8pt, <->, blue] ($(gen.center) + (0, -0.2)$) -- ($(gen.center) + (0, 0.8)$);

    \draw[orange, line width=1.2pt] plot[smooth, tension=0.5] coordinates {
        ($(gen.center) + (-1.6, 0.1)$)
        ($(gen.center) + (-1.1, 0.5)$)
        ($(gen.center) + (-0.6, 0.0)$)
        ($(gen.center) + (0.0, 0.4)$)
        ($(gen.center) + (0.5, -0.1)$)
        ($(gen.center) + (1.1, 0.3)$)
        ($(gen.center) + (1.6, 0.1)$)
    };
    
    \draw[magenta, line width=1.2pt] plot[smooth, tension=0.5] coordinates {
        ($(gen.center) + (-1.3, 0.1)$)
        ($(gen.center) + (-0.8, 0.5)$)
        ($(gen.center) + (-0.3, 0.0)$)
        ($(gen.center) + (0.3, 0.4)$)
        ($(gen.center) + (0.8, -0.1)$)
        ($(gen.center) + (1.4, 0.3)$)
    };

    \node[black, font=\small] at ($(gen.center) + (1.3, 0.6)$) {};
    
    \draw[line width=0.8pt, <->, blue] ($(gen.center) + (-1.6,0.25)$) -- ($(gen.center) + (1.6,0.25)$);
    
    \node[anchor=center, align=center] at ($(gen.center) + (0, -0.6)$) {Signal space \\ generates input};
    \node[anchor=center] at ($(gen.center) + (0, -1.0)$) {};

    \node[anchor=north,align=center] at ($(sys1.north) + (0, -0.2)$) {Quadratic center manifold of FOM\\
    (dimension $n$)};

    \filldraw[fill=orange!30, draw=orange!80!black, thick] 
            ($(sys1.north) + ({-0.9*\r}, {-0.35*\r})$) to[out=-25, in=205] ($(sys1.north) + ({0.9*\r}, {-0.35*\r})$) --
            ($(sys1.north) + ({0.9*\r}, {-1.05*\r})$) to[out=205, in=-25] ($(sys1.north) + ({-0.9*\r}, {-1.05*\r})$) -- cycle;

    \node[anchor=north, align=center] at ($(sys1.north) + (0, -1.65)$) {Quadratic manifold \\ $\p=\mathbf{V}_1\boldsymbol\pi_r + \mathbf{V}_2(\boldsymbol\pi_r\otimes\boldsymbol\pi_r)$};

    \node[anchor=north,align=center] at ($(sys2.north) + (0, -0.1)$) {Linear center manifold of\\  quadratic ROM (dimension $r$)};

    \draw[arrow, black] ($(sys2.north) + (0, 0)$) --($(sys1.south) + (0, 0)$)
    node[midway, left=2pt, align=left, color=black] {$\bV_1,\bV_2$ map ROM\\ center manifold to\\ FOM center manifold};

    \draw[arrow, black] ($(sys1.south) + (0.25, 0)$) --($(sys2.north) + (0.25, 0)$)
    node[midway, right=2pt, align=left, color=black] {Quadratic state \\reduction with \\$\bV_1$ and $\bV_2$};
    
    \filldraw[fill=green!25, draw=green!60!black, thick] 
            ($(sys2.center) + ({-0.85*\r}, 0.7)$) -- ($(sys2.center) + ({0.85*\r}, 0.1)$) --
            ($(sys2.center) + ({0.85*\r}, -1.3)$) -- ($(sys2.center) + ({-0.85*\r}, -0.7)$) -- cycle;

    \node[anchor=south, align=center] at ($(sys2.north) + (0, -2.35)$) {Linear subspace $\boldsymbol \pi_r$\\ of dimension $r \ll n$};

    \draw[blue,line width=0.7pt, <->] ($(comp.center) + (-1.3, 0.25)$) -- ($(comp.center) + (1.3, 0.25)$);
    \draw[blue,line width=0.7pt, <->] ($(comp.center) + (0, -0.3)$) -- ($(comp.center) + (0, 0.8)$);

    \draw[black, line width=1.2pt, shift={(comp.center)}] 
        plot[variable=\t, domain=0:12, samples=200] 
        ({\t/6 - 1.0}, {0.25 + 0.8 * exp(-0.067*\t) * (0.5*sin(\t r) + 0.1*cos(5*\t r))});
        
    \draw[red, dashed, line width=1.2pt, shift={(comp.center)}] 
        plot[variable=\t, domain=0:12, samples=200] 
        ({\t/6 - 1.0}, {0.25 + 0.8 * exp(-0.067*\t) * (0.5*sin(\t r) + 0.1*cos(5*\t r))});
        
    \node[anchor=south, align=center] at ($(comp.south) + (0, 0.1)$) {Matching steady \\ state outputs};

    \draw[arrow] ($(gen.east) + (0, 0.4)$) -- ($(sys1.west) + (0, -0.3)$) 
        node[midway, above, sloped] {Input};
    \draw[arrow] ($(gen.east) + (0, -0.4)$) -- ($(sys2.west) + (0, 0.3)$) 
        node[midway, above, sloped] {Input};

    \draw[arrow] ($(sys1.east) + (0, -0.3)$) -- ($(comp.west) + (0, 0.4)$) 
        node[midway, above, sloped] {Steady output};
    \draw[arrow] ($(sys2.east) + (0, 0.3)$) -- ($(comp.west) + (0, -0.4)$) 
        node[midway, above, sloped] {Steady output};

\end{tikzpicture}
    \caption{Schematic representing the nonlinear moment matching framework using quadratic manifolds. The center manifold of the reduced system (dimension $r$) under the quadratic approximation maps to the center manifold of the original system (dimension $n$) and hence has the same asymptotic steady state output as the original system for inputs generated by signal space. The linear and quadratic matrices $\bV_1$ and $\bV_2$ serve the role of defining the projection matrices for quadratic state reduction.}
    \label{fig:introduction_schematic}
\end{figure}

\begin{itemize}
    \item We introduce an optimization-free algorithm for constructing quadratic projection matrices by solving a sequence of decoupled linear Sylvester equations. In this approach, the choice of the signal space acts as a high-level design specification that can be freely adapted to match the characteristic frequencies or transient behaviors of a target application. Once this operating signal space is selected, the projection bases are uniquely determined with a closed-form expression. Unlike prevailing snapshot-based quadratic manifold techniques, this method requires no subsequent empirical hyperparameter tuning, optimization parameter sweeps, or regularization.

    \item We establish the core theoretical foundation of this framework in Theorem~\ref{thm:moment_matching}. Specifically, we prove that the resulting reduced-order model preserves the nonlinear moments of the full-order system, using the notion of nonlinear moments introduced in~\cite{astolfi2010}. Additionally, we show that the center manifold of the reduced system maps to the quadratic center manifold of the original full-order model (FOM), guaranteeing identical steady-state outputs for signals generated by the specified signal space. This leads to Algorithm~\ref{alg:moment_matching_rom}, which provides a projection-based reformulation of the nonlinear-moment matching framework in the context of quadratic manifolds. 
    
    \item We demonstrate the 
efficacy of the proposed framework on the transport-dominated, one-dimensional damped wave and advection equations. The numerical experiments verify that our quadratic ROM accurately recovers the high-fidelity, full-order steady-state trajectories on the center manifold and thereby, achieves ``nonlinear moment matching". By projecting the high-dimensional state space down to the low dimensional subspace, the framework yields a substantial reduction in simulation time during the ODE solver integration phase, proving that the online savings vastly outweigh the overhead of handling the state-dependent reduced mass matrix. However, state-independent reduced mass matrix can be enforced as in other quadratic manifold approaches. 
\end{itemize}

The remainder of this paper is structured as follows. 
In Section~\ref{sec:background}, we recall relevant matrix and tensor definitions and establish the foundational theory behind nonlinear moment matching, invariance equations, and interpolatory model reduction. 
Section~\ref{sec:qm_rom} introduces the formulation of the quadratic reduced-order model \eqref{eq:rom_state}--\eqref{eq:rom_group} resulting from projecting linear full-order dynamics onto quadratic state approximation. 
In Section~\ref{sec:theoretical}, we establish the core theoretical foundation of the framework. First, we prove that the center manifold of the full-order system under the driving signal generator is a quadratic manifold and derive closed-form expressions for the projection matrices as solutions of decoupled Sylvester equations. Finally, we prove that the quadratic reduced system constructed with these matrices achieves nonlinear moment matching. The proposed computational framework based on this theoretical analysis is summarized in Algorithm~\ref{alg:moment_matching_rom}. 
Section~\ref{sec:signal_space} provides practical guidelines for parameterizing the signal generator to enforce application-specific interpolation conditions. 
Section~\ref{sec:numerical} validates the theoretical guarantees on transport-dominated benchmarks, specifically the 1D advection and damped wave equations, and compares performance against standard linear rational interpolation and existing quadratic manifold techniques.

\section{Background}\label{sec:background}

In this section, we establish the mathematical notation, definitions, and prerequisite theory for linear and nonlinear moment matching. We begin in Section~\ref{sec:notation} by defining the matrix, vector, and tensor operations used throughout the paper. Section~\ref{sec:nonlinearmm} reviews center manifold theory and the notion of nonlinear moments for a dynamical system along with connections with linear moment matching framework in Section~\ref{subsec:linear_moment}. Finally, Section~\ref{sec:moment_matching_def} reviews the formal definitions of moment matching and interpolating reduced-order model.  

\subsection{Notation}
\label{sec:notation}

In this section, we establish the mathematical notation and tensor operations utilized throughout this manuscript. For a review of these standard definitions and properties, the reader is referred to \cite{brewer1978kronecker, graham2018kronecker}. We let $\bI_n$ denote the identity matrix of dimension $n \times n$ (where the subscript may be omitted if the dimension is clear from the context) and $\mathbf{0}$ denote a zero matrix of appropriate dimensions. The Kronecker product of two matrices $\bA \in \mathbb{R}^{m \times n}$ and $\bB \in \mathbb{R}^{p \times q}$ is denoted by $\bA \otimes \bB \in \mathbb{R}^{mp \times nq}$. The Kronecker sum of two square matrices $\bA \in \mathbb{R}^{n \times n}$ and $\bB \in \mathbb{R}^{m \times m}$ is denoted by $\bA \oplus \bB\in \mathbb R^{nm\times nm}$ and defined as
\begin{equation*}
    \bA \oplus \bB = (\bA \otimes \bI_m) + (\bI_n \otimes \bB).
\end{equation*}
To facilitate a compact representation of high-order polynomial and multinomial terms, we utilize a shorthand notation for repeated Kronecker operations. For a square matrix $\bX \in \mathbb{R}^{n \times n}$, its $k$-th Kronecker power $\bX^{(k)}$ and its $k$-th Kronecker sum $\bX \oplus_k \bX$ are defined inductively, for $k \geq 2$, as
\begin{align*}
    \bX^{(k)} &= \bX^{(k-1)} \otimes \bX, \quad \text{with } \bX^{(1)} = \bX, \\
    \bX \oplus_k \bX &= (\bX \oplus_{k-1} \bX) \otimes \bI_n + \bI_n^{(k-1)} \otimes \bX, \quad \text{with } \bX \oplus_1 \bX = \bX.
\end{align*}
Equivalently, the $k$-th Kronecker sum can be expressed explicitly as
\begin{equation}\label{def:kron_sum}
    \bX \oplus_k \bX = \sum_{j=1}^{k} \bI_n^{(j-1)} \otimes \bX \otimes \bI_n^{(k-j)},
\end{equation}
where $\bI_n^{(0)} = 1$. Under this notation, a quadratic state interaction term for a vector $\bw$ simplifies directly to $\bw^{(2)} = \bw \otimes \bw$. Additionally, for matrices of compatible dimensions, the well-known mixed-product property holds
\begin{equation}
    (\bA \bB) \otimes (\bC \bD) = (\bA \otimes \bC)(\bB \otimes \bD).
    \label{eq:kron_mixed_product}
\end{equation}
These algebraic properties are used frequently to simplify expressions throughout this paper.

\subsection{Nonlinear moments and the invariance equation}\label{sec:nonlinearmm}

In the LTI setting, traditional projection-based interpolatory model reduction methods deal with \emph{linear moments}, which are discussed in Sec.~\ref{subsec:linear_moment}. The notion of moment of a dynamical system was reformulated using the theory of steady-state responses and invariant manifolds in \cite{astolfi2010} leading to a unified framework to both redefine the notion of \emph{linear moments} for linear and nonlinear systems and introduce the notion of \emph{nonlinear moments} for both linear and nonlinear systems. To formalize this approach, consider a nonlinear full-order system (FOM) described by
\begin{equation}\label{eq:fom_nonlinear}
    \begin{aligned}
    \dot \bx(t)&=f(\bx(t),\bu(t)),\quad \bx(0)=\bx_0\in\mathbb R^n\\
    \by(t)&=h(\bx(t)),
\end{aligned}
\end{equation} where $\bx(t)\in\mathbb R^n$ represents the state, $\bu(t)\in\mathbb R^m$ represents the input vector, $\by(t)\in\mathbb R^p$ represents the output. The map $f:\mathbb R^n \times \mathbb R^m \rightarrow \mathbb R^n$ describes the nonlinear state evolution and $h: \mathbb R^n \rightarrow \mathbb R^p$ represents the output map for the system. Also consider an autonomous (nonlinear) signal generator defined by
\begin{equation} \label{eq:nonlinear_sig}
    \begin{aligned}
    \dot{\w}(t) &= \mathscr{s}(\w(t)), \quad \w(0)=\w_0\in\mathbb R^r \\
    \by_{d}(t)&=\mathscr l(\w(t)), 
\end{aligned}
\end{equation} where $\w(t)\in\mathbb{R}^r$ represents the state of the signal generator space, $\mathscr s: \mathbb R^r \rightarrow \mathbb R^r$ is a nonlinear function that describes the time-evolution of the signal space and $\mathscr l:\mathbb R^r \rightarrow \mathbb R^m$ is a nonlinear map from the signal generator state, $\w(t)$, to the signal generator output, $\by_d(t)\in\mathbb R^m$. Suppose the nonlinear system in \eqref{eq:fom_nonlinear} is
driven by an input from the autonomous, nonlinear signal generator in \eqref{eq:nonlinear_sig}, i.e., when $\bu(t)=\by_d(t)$. Then the combined dynamics of the two systems \eqref{eq:fom_nonlinear} and \eqref{eq:nonlinear_sig} can be represented by the interconnected system\begin{equation}\label{eq:interconnected_sys}
    \begin{aligned}
    \dot{\w}(t)&=\mathscr{s}(\w(t)),~&\w(0)=\w_0\in\mathbb R^r\\
    \dot{\bx}(t)&= f(\bx(t),\mathscr l(\w(t))),~&\bx(0)=\bx_0\in\mathbb R^n\\\
    \by(t)&=h(\bx(t)).
\end{aligned}
\end{equation}
Under appropriate assumptions, specifically that the unforced system $\dot{\bx} = f(\bx, 0)$ has a locally exponentially stable equilibrium at the origin and the signal generator possesses neutrally stable dynamics (see \cite{astolfi2010}), the interconnected system in \eqref{eq:interconnected_sys} has a locally invariant (center) manifold described by $\bx = \p(\w(t))$. For more details on center manifold theory and its applications see \cite{carr2012applications}.

\begin{defi}[Nonlinear Invariance Equation \cite{astolfi2010}] \label{def:invariance_eqn}
The mapping $\p(\cdot)$, which parameterizes the invariant manifold associated with $(\mathscr s, \mathscr l)$ is the unique local solution to the  partial differential equation (PDE)
\begin{equation} \label{eq:nonlinear_sylvester}
    \frac{\partial \p(\w)}{\partial \w} \mathscr s(\w) = f(\p(\w), l(\w)), \quad \p(0)=0.
\end{equation}
The PDE~\eqref{eq:nonlinear_sylvester} is commonly referred to as the nonlinear invariance equation. 
\end{defi}

\begin{defi}[Nonlinear Moments \cite{astolfi2010}]
The nonlinear moment of the system associated with $(\mathscr s,\mathscr l)$ is defined as the composite mapping $\mathcal{M}_{\mathscr s,\mathscr l}: \mathbb{R}^r \to \mathbb{R}^p$ given by the output of the system restricted to the invariant manifold:
\begin{equation}\label{eq:nonlinear_moment_def}
    \mathcal{M}_{\mathscr s,\mathscr l}(\w) = h\circ\p(\w).
\end{equation}
\end{defi}

We note that the nonlinear moment and center manifold are associated with a signal generator $(\mathscr s,\mathscr l)$, i.e., the center manifold and nonlinear moment, are different for a different choice of signal generator. 

The nonlinear moment matching framework is illustrated in Fig.~\ref{fig:nonlinear_moment_def}, where the nonlinear moment is a mapping from the signal generator to the steady-state output of the system. The left panel defines the signal space, where the signal generator is governed by the dynamics $\dot{\w} = \mathscr s(\w)$ and generates the input $\bu = \mathscr l(\w)$. The central panel depicts the center manifold $\p(\w)$, which is an invariant geometric surface defined as the solution of the nonlinear invariance equation \eqref{eq:nonlinear_sylvester}. When the system is driven by this signal generator, an arbitrary state trajectory $\bx(t)$ (shown in blue) exhibits initial transient dynamics before eventually converging onto this manifold. The trajectory strictly restricted to this manifold represents the exact steady-state behavior of the system, denoted by $\bx_{ss}(t) = \p(\w(t))$ (shown in black). Finally, the right panel illustrates the output converging to its steady state as the transient decays, matching the steady-state output generated by passing the manifold dynamics through the output function $h(\cdot)$. The overarching nonlinear moment, defined as $\mathcal{M}_{s,l}(\w) = h \circ \p(\w)$, fundamentally captures this relationship, directly linking the signal space to the steady-state output while entirely bypassing the transient phase. 

\begin{figure}
    \centering
    \begin{center}
\resizebox{\textwidth}{!}{
\begin{tikzpicture}
    \node[
        draw=orange!80!black, 
        line width=2pt, 
        inner sep=10pt, 
        align=center, 
        fill=white,
        minimum width=5.5cm,
        minimum height=4.5cm, font=\large
    ] (leftbox) at (0,0) {
    \underline{\textbf{Nonlinear signal generator}}\\ \\
        $\dot{\w}(t) = \mathscr s(\w(t))$ \\[0.5em]
        $\bu(t) = \mathscr l(\w(t))$ \\[0.5em]
        $\w(0) = \w_0 \in \mathbb{R}^r$ \\[0.5em]
        $(\mathscr s, \mathscr l,\w_0)$
    };

    \begin{axis}[
        name=middleaxis,
        at={(leftbox.east)},
        anchor=west,
        xshift=2cm, 
        axis lines=none,
        width=7cm, height=6cm, 
        view={0}{15}, 
        clip=false 
    ]
        
        \addplot3[
            surf, color=yellow, faceted color=orange, 
            domain=-2:1.5, domain y=0:1
        ] {(x^2+y^2)};

        \addplot3[
            black, thick, samples=50, 
            domain=-1.5:1, y domain=0:0, samples y=1    
        ] (x,1/2,{1/4+x^2});

        \addplot3[
            blue, thick, samples=50, 
            domain=-1.5:1, y domain=0:0, samples y=1  
        ] (x,1/2,{0.001*exp(1.5*(2-2*x)) + 1/4+x^2});

        \node at (axis cs: -0.25, 0.5, 0.2) [anchor=north, font=\small] {\color{black} $\mathbf x_{\text{ss}}(t)$};

        \node at (axis cs: -1.45, 0.5, 4.5) [anchor=east, font=\small, blue] {$\mathbf x(t)$};

        \node at (axis cs: 0, 0.5, 5.0) [
            draw=orange!40!white, fill=white, inner sep=6pt, 
            align=left, line width=1pt
        ] {
            \small \color{black!80!black} $  \frac{\partial }{\partial t} \boldsymbol\pi(\omega)=f(\boldsymbol \pi(\omega),\mathscr l(\omega))$ 
        };
    \end{axis}

    \coordinate (rightboxstart) at ([xshift=2cm]middleaxis.east);

    \begin{scope}[shift={(rightboxstart)}, local bounding box=rightbox]
        
        \draw[orange!80!black, line width=2pt] (0, -2.25) rectangle (5.5, 2.25);

        \draw[->, line width=1.2pt] (0.3, 0) -- (5.2, 0);   
        \draw[<->, line width=1.2pt] (0.8, -1.8) -- (0.8, 1.8);

        \draw[blue!80!black, line width=1.5pt, smooth, samples=100, domain=0.8:5.0] 
            plot (\x, {0.6 * sin((\x - 0.8) * 250)});
            
        \draw[orange!90!black, line width=1.5pt, smooth, samples=100, domain=0.8:5.0] 
            plot (\x, {(0.6+(0.4*exp(-0.5*\x))) * sin((\x - 0.8) * 250)});
    \end{scope}

    \draw[->, line width=1.5pt, shorten >=0.2cm, shorten <=0.2cm] 
        (leftbox.east) -- node[midway, above, font=\normalsize] {\Large$ \boldsymbol{\pi}(\w)$} (middleaxis.west |- leftbox.east);
    
    \draw[->, line width=1.5pt, shorten >=0.2cm, shorten <=0.2cm] 
        (middleaxis.east |- leftbox.east) -- node[midway, above, font=\normalsize] {\Large $h(\cdot)$} (rightbox.west |- leftbox.east);

    \draw[->, line width=1.5pt, rounded corners=10pt, shorten >=0.1cm, shorten <=0.1cm]
        (leftbox.north) 
        -- (leftbox.north |- 0, 3) 
        -- (rightbox.north |- 0, 3) node[midway, above, font=\large] {$\mathcal{M}_{\mathscr s,\mathscr l}(\w)=h\circ \boldsymbol{\pi}(\w)$} 
        -- (rightbox.north);

    \path (middleaxis.west) -- (middleaxis.east) coordinate[midway] (middlecenter);
    
    \node[font=\normalsize] at (leftbox.center |- 0, -2.67) {(Signal space ($\mathscr s,\mathscr l,\omega_0$))};
    \node[font=\normalsize] at (middlecenter |- 0, -2.67) {(Nonlinear center manifold $\boldsymbol \pi(\omega)$)};
    \node[font=\normalsize] at (rightbox.center |- 0, -2.67) {(Steady state output)};

\end{tikzpicture}}
\end{center}
    \caption{Schematic representation of the nonlinear moment framework.}
    \label{fig:nonlinear_moment_def}
\end{figure}
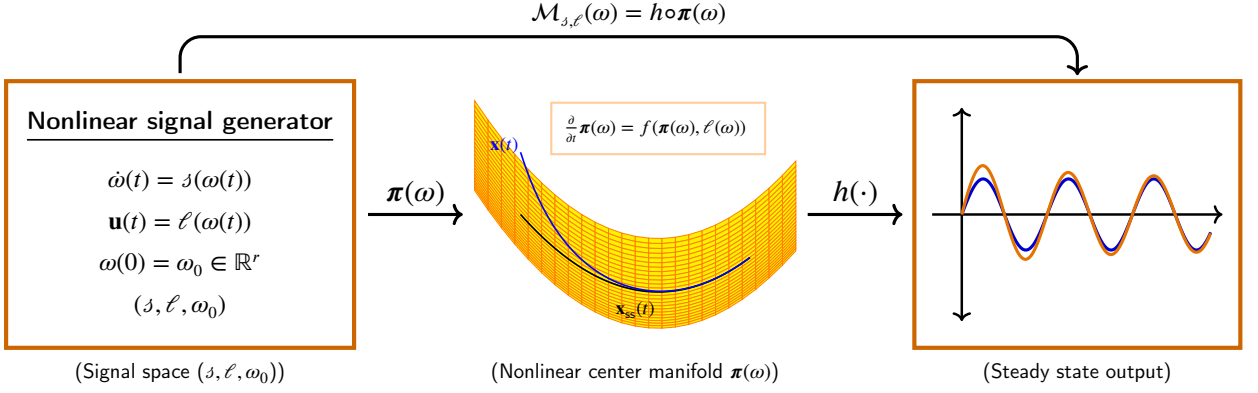

\subsubsection{Formulations of moments for linear systems} \label{subsec:linear_moment}

In this section, we consider the special case when $f(\bx,\bu)=\bA\bx+\bB\bu$ and $g(x)=\bC\bx$ where the state-space matrices $\bA\in\mathbb R^{n\times n},\bB \in \mathbb R^{n\times m}$ and $\bC\in\mathbb R^{p\times n}$ are constant. Under this choice the full-order model (FOM) is described by
\begin{equation}\label{eq:lti_sys} 
    \begin{aligned}
\dot{\bx}(t) &= \bA\bx(t) + \bB\bu(t),\quad \bx(0)=\bx_0\\
\by(t) &=\bC\bx(t),
\end{aligned}
\end{equation}
where $\bx(t) \in \mathbb{R}^n$ represents the state, $\bu(t)\in\mathbb R^m$ represents the input, and $\by(t)\in \mathbb R^p$ represents the output. {Assuming zero initial conditions ($\bx(0) = \mathbf{0}$), taking the Laplace transform converts the differential equations ~\eqref{eq:lti_sys} into an algebraic equation that reads $\bY(s)=\bC(s\mathbf{I} - \bA)^{-1}\bB\bU(s)$ where $\bY(s)$ and $\bU(s)$ denote the Laplace transforms of the output $\by(t)$ and input $\bu(t)$, respectively. The function $\bH(s) = \bC(s\mathbf{I} - \bA)^{-1}\bB\in\mathbb R^{p\times m}$ denotes the associated transfer function of the LTI system and represents the input-output behaviour of the system in the frequency domain. In this LTI setting, the transfer function is a rational function of $s$.}

\begin{defi}[Linear Moments~\cite{ACA05}]\label{def:moments}
Let $\bH(s)$ denote the transfer function of the linear time-invariant (LTI) system in \eqref{eq:lti_sys}. The $0$-th linear moment at a frequency $s^* \in \mathbb{C}$ is defined as the transfer function evaluated at that point, denoted by $\eta_0(s^*) = \bH(s^*)$. For any integer $k \geq 1$, the $k$-th linear moment at $s^*$ is defined as the $k$-th derivative of the transfer function evaluated at $s^*$, given by
\begin{equation*}
    \eta_k(s^*) =(-1)^k \left. \frac{d^k}{ds^k} \bH(s) \right|_{s=s^*}.
\end{equation*}
\end{defi}

Linear moment matching methods construct reduced systems, i.e., systems of the form~\eqref{eq:lti_sys} with a reduced state dimension, whose now lower-degree rational transfer function interpolates the moments of the original linear system as defined in Definition~\ref{def:moments}.  {This frequency-domain formulation is standard in interpolatory projection-based model reduction framework~\cite{morAntBG20}.} Astolfi~\cite{astolfi2010} demonstrated that these linear moments can be equivalently characterized in the time domain as the steady-state output response generated by an autonomous, linear signal generator of dimension $r$,\begin{equation}\label{eq:lin_signal_gen}
\begin{aligned}
    \dot{\w}(t) &= \bS\w(t), \quad \w(0) = \w_0 \in \mathbb{R}^r,\\
    \bu(t) &= \bL\w(t).
\end{aligned}    
\end{equation} where $\bS\in\mathbb R^{r\times r}$, $\bL \in\mathbb R^{m\times r}$ are matrices chosen such that the pair $\bS$ and $\bL$ are locally 
\emph{observable}, i.e., when \begin{align*}
    \operatorname{rank}\left(\left[\begin{array}{c}
        \bL \\ \bL\bS\\ \vdots \\ \bL\bS^{r-1}
    \end{array}\right]\right)=r.
\end{align*} Intuitively, this condition ensures that the input $\bu(t)$ depends on every component of $\w(t)$ with no hidden or redundant components in the signal state. Similar to \eqref{eq:interconnected_sys}, we consider the interconnection of the LTI system in \eqref{eq:lti_sys} and linear signal generator in \eqref{eq:lin_signal_gen}, which leads to the interconnected system\begin{equation}\label{eq:conected_lin_sys}
    \begin{aligned}
    \dot{\w}(t)&=\bS\w(t),~&\w(0)=\w_0\in\mathbb R^r\\
    \dot{\bx}(t)&= \bA\bx(t)+\bB\mathscr \bL\w(t),~&\bx(0)=\bx_0\in\mathbb R^n\\\
    \by(t)&=\bC\bx(t).
\end{aligned}
\end{equation} In this linear setting, the invariance equation~\eqref{eq:nonlinear_sylvester} reduces to a linear matrix equation \eqref{eq:linear_sylvester}. The center manifold $\p(\w)=\p\w$ where $\p\in\mathbb R^{n\times r}$ is a linear subspace obtained as the solution of the Sylvester equation 
\begin{equation} \label{eq:linear_sylvester}
    \bA\p + \bB\bL = \p\bS.
\end{equation} Moreover, the moment $\mathcal M_{\bS,\bL}(\cdot)$ associated with $(\bS,\bL)$ in the sense of \cite{astolfi2010} is given as \begin{align*}
    \mathcal M_{\bS,\bL}(\w)=\bC \p(\w) = \bC \p \w. 
\end{align*} To see how this fits into the classical projection-based framework in \cite{morAntBG20}, consider the case where we want to interpolate the transfer function $\bH(s)$ at $r$ distinct complex points $\{\mu_1, \mu_2, \dots, \mu_r\}$ along the tangential directions $\{\ell_1,\dots, \ell_r \}$. In the interpolatory projection-based framework, this is referred to as one-sided tangential interpolation where we construct a projection matrix $\mathbf{V} \in \mathbb{R}^{n \times r}$ whose columns span the Krylov subspace,
\begin{equation}
    \text{Im}(\mathbf{V}) = \text{span}\left\{ (\mu_1\mathbf{I} - \bA)^{-1}\bB\ell_1, \dots, (\mu_r\mathbf{I} - \bA)^{-1}\bB\ell_r \right\}.
\end{equation}
If we choose the signal generator matrix $\bS$ to be diagonal, $\bS = \text{diag}(\mu_1, \dots, \mu_r)$, and set $\bL = [\ell_1, \dots, \ell_r]$ such that the columns correspond to tangential directions in $\mathbb R^m$, the Sylvester equation \eqref{eq:linear_sylvester} can be solved column-by-column. For the $i$-th column $\p_i$, the equation yields:
\begin{equation}
    \bA\p_i + \bB\ell_i = \mu_i \p_i \implies \p_i = (\mu_i \mathbf{I} - \bA)^{-1}\bB\ell_i.
\end{equation}
Thus, the columns of the steady-state mapping matrix $\p$ are precisely the exact vectors that form the interpolating subspace in the projection-interpolatory framework, $\boldsymbol\Pi=\bV$. Thus, the moment $\mathcal M_{\bS,\bL}(\cdot)$ reduces to
\begin{equation}\label{eq:TF_interpolation}
    \mathcal M_{\bS,\bL}(\w)=\bC\p\w = [\bH(\mu_1)\ell_1\quad\bH(\mu_2)\ell_2\quad \dots\quad  \bH(\mu_r)\ell_r]\w
\end{equation}  From \eqref{eq:TF_interpolation}, we see that the linear moment (and the steady state response) for an LTI system is determined by transfer function evaluations at the frequencies $\mu_1,\dots,\mu_r$ and the nonlinear moment matching framework in \cite{astolfi2010} boils down to the classical interpolatory projection framework~\cite{morAntBG20}.

\begin{figure}
    \centering
    \begin{center}
\resizebox{\textwidth}{!}{
\begin{tikzpicture}

   \node[
        draw=orange!80!black, 
        line width=2pt, 
        inner sep=10pt, 
        align=center, 
        fill=white,
        minimum width=5.5cm,
        minimum height=4.5cm, font=\large
    ] (leftbox) at (0,0) {
        \underline{\textbf{Linear signal generator}}\\ \\
        $\dot{\w}(t) =  \bS\w(t),$ \\[0.5em]
        $\bu(t) = \bL\w(t),$ \\[0.5em]
        $\w(0)=\w_0\in\mathbb R^r$
    };

   \begin{axis}[
        name=middleaxis,
        at={(leftbox.east)},
        anchor=west,
        xshift=2cm, 
        axis lines=none,
        width=7cm, height=6cm, 
        view={0}{15}, 
        clip=false 
    ]
        
        \addplot3[
            surf, color=yellow!80!white, faceted color=orange!80!black, 
            domain=-2:1.5, domain y=0:1
        ] {0.8*x + 1.2*y + 1.5};

        \addplot3[
            black, thick, samples=50, 
            domain=-1.5:1, y domain=0:0, samples y=1    
        ] (x, 0.5, {0.8*x + 2.1});

         \addplot3[
            blue, thick, samples=50, 
            domain=-1.5:1, y domain=0:0, samples y=1  
        ] (x, 0.5, {0.8*x + 2.1 + 0.001*exp(1.5*(2-2*x))});

         \node at (axis cs: -0.25, 0.5, 1.3) [anchor=north, font=\small] {\color{black} $\mathbf x_{\text{ss}}(t)$};

        \node at (axis cs: -1.45, 0.5, 2.5) [anchor=east, font=\small, blue] {$\mathbf x(t)$};

        \node at (axis cs: -1, 0.5, 4.2) [
            draw=orange!40!white, fill=white, inner sep=6pt, 
            align=left, line width=1pt
        ] { Invariance equations: \\
            \small \color{black!80!black} $ \p\bS=\bA\p + \bB \bL$ 
        };
    \end{axis}

   \coordinate (rightboxstart) at ([xshift=2cm]middleaxis.east);

    \begin{scope}[shift={(rightboxstart)}, local bounding box=rightbox]
        
        \draw[orange!80!black, line width=2pt] (0, -2.25) rectangle (5.5, 2.25);

        \draw[->, line width=1.2pt] (0.3, 0) -- (5.2, 0);   
        \draw[<->, line width=1.2pt] (0.8, -1.8) -- (0.8, 1.8);

        \draw[blue!80!black, line width=1.5pt, smooth, samples=100, domain=0.8:5.0] 
            plot (\x, {0.6 * sin((\x - 0.8) * 250)});
            
        \draw[orange!90!black, line width=1.5pt, smooth, samples=100, domain=0.8:5.0] 
            plot (\x, {(0.6+(0.4*exp(-0.5*\x))) * sin((\x - 0.8) * 250)});
    \end{scope}

      \draw[->, line width=1.5pt, shorten >=0.2cm, shorten <=0.2cm] 
        (leftbox.east) -- node[midway, above, font=\normalsize] {\Large$ \boldsymbol{\pi}(\w)$} (middleaxis.west |- leftbox.east);
    
    \draw[->, line width=1.5pt, shorten >=0.2cm, shorten <=0.2cm] 
        (middleaxis.east |- leftbox.east) -- node[midway, above, font=\normalsize] {\Large $\bC\bx$} (rightbox.west |- leftbox.east);

     \draw[->, line width=1.5pt, rounded corners=10pt, shorten >=0.1cm, shorten <=0.1cm]
        (leftbox.north) 
        -- (leftbox.north |- 0, 3) 
        -- (rightbox.north |- 0, 3) node[midway, above, font=\large] {$\mathcal{M}_{\bS,\bL}(\w)= \bC\p\w$} 
        -- (rightbox.north); 

      \path (middleaxis.west) -- (middleaxis.east) coordinate[midway] (middlecenter);
    
    \node[font=\normalsize] at (leftbox.center |- 0, -2.67) {(Signal space ($\bS,\bL,\omega_0$))};
    \node[font=\normalsize] at (middlecenter |- 0, -2.67) {(Linear subspace of dimension $r$)};
    \node[font=\normalsize] at (rightbox.center |- 0, -2.67) {(Steady state output)};

\end{tikzpicture}}
\end{center}
    \caption{{Schematic representation of linear moment for LTI systems.}}
    \label{fig:linear_moment_def}
\end{figure}
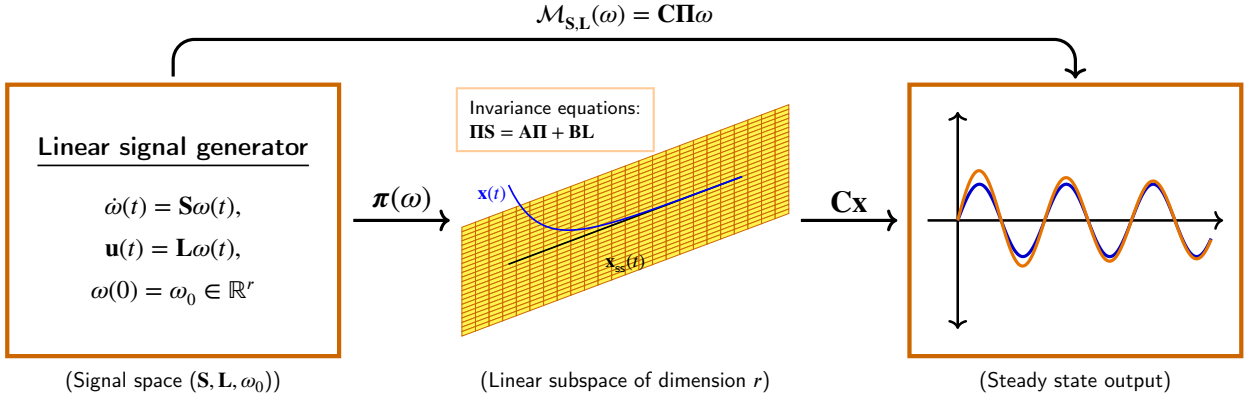

{When both the system and the signal generator are linear, the nonlinear moment-matching framework in Fig.~\ref{fig:nonlinear_moment_def} boils down to the linear framework depicted in Fig.~\ref{fig:linear_moment_def}. In this setting, the nonlinear invariance equation \eqref{def:invariance_eqn} reduces to the linear Sylvester equation~\eqref{eq:linear_sylvester}, causing the center manifold to simplify to an $r$-dimensional linear subspace. Due to the invariance property of this subspace, trajectories originating on it remain confined to it for all time. State trajectories starting off on the subspace (such as $x(t)$, shown in blue) undergo an initial transient before decaying onto the subspace to yield the steady-state state trajectory $\bx_{ss}(t)=\bC\p\w(t)$. Consequently, the system output exhibits a corresponding transient phase before converging to its steady-state response. The overarching moment mapping $\mathcal M_{\bS,\bL}(\w)=\bC\p\w$ directly links the signal space to this steady-state output, which, for appropriate choices of $(\bS,\bL)$, simplifies to classical transfer function interpolation as shown in \eqref{eq:TF_interpolation}.}

\begin{rem}\label{rem:choice_SL}
    Thus, in the linear setting with the above choice of signal generator,  the matrices $\bS$ and $\bL$ in the nonlinear moment matching framework act as the direct analogs to the interpolation points $\mu_i$ and tangential directions $\ell_i$ in the classical projection-interpolatory framework in \cite{morAntBG20}. Because the steady-state dynamics restricted to the manifold $\p$ are driven purely by the signal generator (i.e., $\dot{\bx} = \p\bS\w$), a reduced-order model constructed with these matrices intrinsically achieves the desired one-sided tangential interpolation. Consequently, throughout this paper, we treat the signal generator not merely as a fixed input structure, but as tunable design parameters that can be explicitly chosen to enforce targeted interpolation requirements and steady-state fidelity in the reduced system.
\end{rem}
\subsubsection{Moment-matching (interpolatory) reduced systems} \label{sec:moment_matching_def}
This section reviews the formal definition of moment matching for nonlinear systems and establishes the general structure of an interpolatory reduced-order model (ROM). Consider the full-order system described in \eqref{eq:fom_nonlinear}, and let a candidate reduced system be given by
\begin{equation}\label{eq:nonlinear_rom}
    \begin{aligned}
    \dot{\hat{\bx}}(t) &= \hat{f}(\hat{\bx}(t),\bu(t)),\quad  \hat{\bx}(0)=\hat{\bx}_0,\\
    \by_r(t) &= \hat{h}(\hat{\bx}(t)),
    \end{aligned}
\end{equation} 
where $\hat{\bx}(t)\in\mathbb{R}^r$ represents the reduced state vector with $r<n$, $\bu(t)\in\mathbb{R}^m$ is the input, and $\by_r(t)\in\mathbb{R}^{p}$ is the ROM output.  The map $\hat{f}:\mathbb{R}^r\times \mathbb{R}^m\rightarrow\mathbb{R}^r$ governs the reduced state evolution, and $\hat{h}: \mathbb{R}^r \rightarrow \mathbb{R}^p$ defines the output map. Suppose the system in \eqref{eq:nonlinear_rom} is driven by the signal generator $(\mathscr{s},\mathscr{l})$, yielding the interconnected system:
\begin{equation}\label{eq:inter_mm}
    \begin{aligned}
    \dot{\w}(t) &= \mathscr{s}(\w(t)), \quad &\w(0) &= \w_0,\\
    \dot{\hat{\bx}}(t) &= \hat{f}(\hat{\bx}(t),\mathscr{l}(\w(t))),\quad  &\hat{\bx}(0) &= \hat{\bx}_0,\\
    \by_r(t) &= \hat{h}(\hat{\bx}(t)).
    \end{aligned}
\end{equation}
This interconnected system possesses its own center manifold, denoted by $\boldsymbol{\pi}_r(\w)$, which is obtained by solving the corresponding invariance equation in Definition~\ref{def:invariance_eqn}. The nonlinear moment of the reduced system at $(\mathscr{s},\mathscr{l})$ is then given as
\begin{equation}
    \widehat{\mathcal{M}}_{\mathscr{s},\mathscr{l}}(\w) = \hat{h}\circ\boldsymbol{\pi}_r(\w).
\end{equation}

\begin{defi}[Nonlinear moment matching \cite{astolfi2010}]
 The system in \eqref{eq:nonlinear_rom} achieves nonlinear moment matching of the full-order system \eqref{eq:fom_nonlinear} at $(\mathscr{s},\mathscr{l})$ if, for all $\w$ in a neighborhood of the origin, the moments coincide, i.e.,
\begin{equation}\label{eq:mm_general}
    \widehat{\mathcal{M}}_{\mathscr{s},\mathscr{l}}(\w) = \hat{h}\circ\boldsymbol{\pi}_r(\w) = h\circ \boldsymbol{\pi}(\w) = \mathcal{M}_{\mathscr{s},\mathscr{l}}(\w).
\end{equation} 
{Additionally, the system in \eqref{eq:nonlinear_rom} is formally considered a reduced-order model of \eqref{eq:fom_nonlinear} if $r < n$.}
\end{defi}

In \cite{astolfi2010}, the author introduce a parameterized family of reduced models that inherently match the moments of \eqref{eq:fom_nonlinear} at $(\mathscr{s},\mathscr{l})$. This family is defined by
\begin{equation}
    \begin{aligned}
    \dot{\bx}_r(t) &= \mathscr{s}(\bx_r(t)) - \delta(\bx_r(t))\mathscr{l}(\bx_r(t)) + \delta(\bx_r(t))\bu(t), \\
    \by_r(t) &= h(\boldsymbol{\pi}(\bx_r(t))),
    \end{aligned}
\end{equation}
where $\bx_r \in \mathbb{R}^r$ is the reduced state vector, and $\delta(\cdot)$ is an arbitrary mapping chosen such that the associated invariance equation,
\begin{equation} \label{eq:astolfi_rom}
    \frac{\partial p(\w)}{\partial \w}\mathscr{s}(\w) = \mathscr{s}(p(\w)) - \delta(p(\w))\mathscr{l}(p(\w)) + \delta(p(\w))\mathscr{l}(\w),
\end{equation}
admits the unique solution $p(\w) = \w$. Satisfying this condition guarantees that the center manifold of the reduced system collapses to the identity map ($\boldsymbol{\pi}_r(\w) = \w$). Consequently, the reduced model evaluates the exact full-order output map along the FOM's center manifold, ensuring that the ROM perfectly replicates the steady-state output response of the high-dimensional plant for the prescribed class of inputs.

\begin{rem}
    {For a given FOM and chosen signal generator, there exist an infinite number of reduced systems that satisfy the moment-matching condition. Furthermore, the FOM and ROM are not required to share the same structural form; a linear FOM may be reduced to a nonlinear ROM, and vice versa. The fundamental requirement is simply that the ROM dimension is strictly smaller than the FOM dimension ($r \ll n$) while fulfilling the exact moment-matching property in \eqref{eq:mm_general}. In this paper, we present a method of constructing a quadratic ROM that interpolates \emph{the nonlinear moments} of a linear (LTI) full order system.}
\end{rem}

\section{Defining the quadratic reduced system}\label{sec:qm_rom}

We consider the problem of constructing reduced systems for linear time-invariant (LTI) systems of the form \begin{equation}\label{eq:lti_fom}
    \begin{aligned}
    \dot\bx(t
    )&=\bA\bx(t
    )+ \bB\bu(t),\quad  \bx(0)=\bx_0\in\mathbb R^n\\
    \by(t)&=\bC\bx(t),
\end{aligned}
\end{equation} where $\bx(t)\in\mathbb R^n$ represents the state, $\bu(t)\in\mathbb R^{m}$ represents the input and $\by(t)\in\mathbb R^p$ represents the output and the state matrices are $\bA\in\mathbb R^{n\times n},\bB \in \mathbb R^{n\times m}$ and $\bC\in\mathbb R^{p\times n}$. In this paper, we focus on quadratic approximations, where the state $\mathbf{x}(t)\in\mathbb R^n$ is approximated as
\begin{equation} \label{eq:quad_approx}
\mathbf{x}(t) \approx \boldsymbol{\Phi}(\mathbf{x}_r(t)) = \mathbf{V}_1 \mathbf{x}_r(t) + \mathbf{V}_2(\mathbf{x}_r(t) \otimes \mathbf{x}_r(t)), 
\end{equation}
where $\mathbf{x}_r(t) \in \mathbb{R}^r$ is the reduced state with $r \ll n$. The matrices $\mathbf{V}_1 \in \mathbb{R}^{n \times r}$ and $\mathbf{V}_2 \in \mathbb{R}^{n \times r^2}$ define the linear and quadratic components of the manifold, respectively. Substituting the quadratic approximation \eqref{eq:quad_approx} into the linear FOM \eqref{eq:lti_fom} yields the  residual equation 
\begin{align*}
    \mathbf{V}_1 \frac{d\mathbf{x}_r}{dt} + \mathbf{V}_2 \frac{d}{dt}(\mathbf{x}_r(t) \otimes \mathbf{x}_r(t)) = \mathbf{A} \mathbf{V}_1 \mathbf{x}_r(t) + \mathbf{A} \mathbf{V}_2 (\mathbf{x}_r(t) \otimes \mathbf{x}_r(t)) + \mathbf{B} \mathbf{u}(t) +\mathbf r(t),
\end{align*}
where $\mathbf r(t)\in\mathbb R^n$ is the residual vector. Within the framework of classical interpolatory MOR, we introduce a test projection matrix $\bW\in\mathbb R^{n\times r}$  such that $\bW^\top\bV_1$ is invertible and $\mathbf{W}^\top \mathbf{r}(t) = \mathbf{0}$. In other words, we solve the FOM along the columns of $\bW$. Thus, projecting the residual equation with $\bW^\top$ from the left yields
\begin{align*}
    \mathbf{W}^\top \mathbf{V}_1 \frac{d\mathbf{x}_r}{dt} + \mathbf{W}^\top \mathbf{V}_2 \frac{d}{dt}(\mathbf{x}_r \otimes \mathbf{x}_r) &= \mathbf{W}^\top \mathbf{A} \mathbf{V}_1 \mathbf{x}_r + \mathbf{W}^\top \mathbf{A} \mathbf{V}_2 (\mathbf{x}_r \otimes \mathbf{x}_r) + \mathbf{W}^\top \mathbf{B} \mathbf{u}.
\end{align*}
This formulation results in a reduced system given by
\begin{subequations}\label{eq:rom_state}
\begin{align}
   (\bI_r+\bE_r(\mathbf{x}_r\oplus \bx_r) )\frac{d\mathbf{x}_r(t)}{dt} &= \mathbf{A}_r \mathbf{x}_r(t) + \mathbf{H}_r (\mathbf{x}_r(t) \otimes \mathbf{x}_r(t)) + \mathbf{B}_r \mathbf{u}(t) \\
    \mathbf{y}_r(t) &= \widehat h(\bx_r)=\mathbf{C}_r \mathbf{x}_r(t) + \mathbf{K}_r (\mathbf{x}_r(t) \otimes \mathbf{x}_r(t)), 
\end{align}
\end{subequations}
where the reduced matrices and the state-dependent reduced mass matrix term denoted by $\mathbf{M}_r(\mathbf{x}_r)$ are defined as
\begin{subequations}\label{eq:rom_group}
\begin{align}
    \mathbf{A}_r &= (\mathbf{W}^\top \mathbf{V}_1)^{-1}\mathbf{W}^\top \mathbf{A} \mathbf{V}_1, \quad \mathbf{B}_r = (\mathbf{W}^\top \mathbf{V}_1)^{-1}\mathbf{W}^\top \mathbf{B}, \quad \mathbf{C}_r = \mathbf{C} \mathbf{V}_1 \quad \mathbf{H}_r = (\mathbf{W}^\top \mathbf{V}_1)^{-1}\mathbf{W}^\top \mathbf{A} \mathbf{V}_2,\\
     \mathbf{E}_r &= (\mathbf{W}^\top \mathbf{V}_1)^{-1}\mathbf{W}^\top\mathbf{V}_2, \qquad \mathbf{M}_r(\mathbf{x}_r) = \bI_r + \bE_r (\mathbf{I}_r \otimes \mathbf{x}_r + \mathbf{x}_r \otimes \mathbf{I}_r) \quad \mathbf{K}_r = \mathbf{C} \mathbf{V}_2.\label{eq:rom_b}
\end{align}
\end{subequations}

In the proposed framework of this paper, the choice of the left projection basis $\bW$ is uncoupled and independent of the choice of $\bV_1,\bV_2$. This flexibility allows one to tailor the algebraic properties of the reduced-order model to specific computational requirements. For example, by enforcing the orthogonality condition, $\mathbf{W}^\top \mathbf{V}_2 = \mathbf{0}$ (like in \cite{geelen2023operator,QM_framework} for example), $\mathbf{M}_r$ simplifies to the identity matrix $\bI_r$. This is particularly advantageous as it simplifies the reduced dynamics, ensuring that the resulting ROM avoids state dependent terms multiplying the derivative of the reduced state. In our framework, we consider the more general case of where $\bW^\top\bV_2\neq 0$, while naturally retaining both the orthogonalized formulation and the classical Galerkin projection $\bW=\bV_1$ as special cases. We note that the mass-matrix term like in~\eqref{eq:rom_state} also appears in the context of nonlinear decoder approximation in \cite{benner2023quadratic}.\\ 

\noindent\fbox{%
    \parbox{\dimexpr\linewidth-2\fboxsep-2\fboxrule}{%
        In the next section, we provide a system-theoretic way of choosing these projection matrices $\bV_1,\bV_2$ and $\bW$ such that the ROM defined in \eqref{eq:rom_state}--\eqref{eq:rom_group} preserves the nonlinear moments of the FOM defined in \eqref{eq:lti_fom}.
    }%
}\\

\section{Constructing interpolatory ROM using quadratic approximations} \label{sec:theoretical}

This section establishes the theoretical foundations for constructing an interpolating reduced-order model (ROM) that achieves moment matching over a prescribed signal space. We begin by characterizing the center manifold obtained as the solution to the invariance equation under a specialized class of signal generators. Specifically, Theorem~\ref{thm:quad_pi} demonstrates that this center manifold is quadratic, providing theoretical motivation for the choice of the quadratic projection matrices $\bV_1$ and $\bV_2$. These developments culminate in  Theorem~\ref{thm:moment_matching}, the primary theoretical result of this work, which states that the ROM defined in \eqref{eq:rom_state} and \eqref{eq:rom_group} matches the moments of the full-order system in the sense of \cite{astolfi2010}, given an appropriate selection of $\bV_1$, $\bV_2$, and $\bW$. To facilitate readability, the proof of this main result is structured into two sequential steps. First, Theorem~\ref{thm:identity_manifold} establishes that the ROM constructed with these specific projection matrices guarantees that its own center manifold reduces to the identity map. Second, the overarching moment-matching result is achieved by combining the quadratic center manifold property from Theorem~\ref{thm:quad_pi} with the identity mapping property from Theorem~\ref{thm:identity_manifold}.

\subsection{Quadratic Center Manifolds for Model Reduction} 
\label{subsec:quad_center_manifolds}

We start by introducing an autonomous signal generator whose state $\omega(t) \in \mathbb{R}^r$ evolves linearly, while its output $\mathbf{u}(t)$ incorporates both linear and quadratic terms. Let $\omega_0 = \omega(0) \in \mathbb{R}^r$ denote the initial state of the generator. Suppose the signal generator dynamics are chosen as
\begin{equation}\label{eq:quad_signal_gen}
    \begin{aligned}
    \dot{\omega}(t) &= \mathbf{S}\omega(t),\quad\quad \w(0)=\w_0, \\
    \mathbf{u}(t) &= \mathbf{L}(\omega(t)) = \mathbf{L}_1 \omega(t) + \mathbf{L}_2 \omega^{(2)}(t), 
    \end{aligned}
\end{equation}
where $\mathbf{S}\in\mathbb{R}^{r\times r}$ is a full-rank matrix, $\mathbf{L}_1\in\mathbb{R}^{m\times r}$ and  $\mathbf{L}_2\in\mathbb{R}^{m\times r^2}$. Consequently, the explicit, time-dependent input that will be fed into the FOM system  \eqref{eq:lti_fom} is given by
\begin{equation}\label{eq:input_explicit}
    \mathbf{u}(t) = \mathbf{L}_1 e^{\mathbf{S}t}\omega_0 + \mathbf{L}_2 \left(e^{\mathbf{S}t} \otimes e^{\mathbf{S}t}\right)(\omega_0\otimes\omega_0).
\end{equation}
To provide intuition for the upcoming formulation, we note that selecting the signal generator parameters $\bS$, $\bL_1$, and $\bL_2$ in \eqref{eq:quad_signal_gen} is functionally equivalent to choosing the interpolation frequencies and tangential directions, much like discussed in Section~\ref{subsec:linear_moment}. Consequently, the quadratic ROM constructed later in this section explicitly \emph{interpolates} the exact nonlinear moments of the full-order system at these targeted frequencies and directions. We rigorously formalize this connection between choice of signal generators and nonlinear moment interpolation and provide canonical choices for selecting these parameters in Section~\ref{sec:signal_space}. Assuming an interconnection of  this signal space \eqref{eq:quad_signal_gen} with the LTI system from \eqref{eq:lti_fom}, we obtain
\begin{equation}\label{eq:interconnection_quad_lti}
    \begin{aligned}
    \dot \w(t) &= \bS\w(t), \quad \w(0)=\w_0,\\
    \dot\bx(t
    )&=\bA\bx(t
    )+ \bB(\bL_1\w(t) + \bL_2 (\w(t)\otimes \w(t))),\quad  \bx(0)=\bx_0\in\mathbb R^n,\\
    \by(t)&=\bC\bx(t).
\end{aligned}
\end{equation}

This interconnected system admits a center invariant manifold which, as discussed in Section~\ref{sec:background}, is given as the solution of the invariance equation. For this interconnected system,  the invariance equation has the form
\begin{equation}\label{eq:invariance}
    \frac{\partial \Pi(\omega)}{\partial \omega} \mathbf{S}\omega = \mathbf{A} \Pi(\omega) + \mathbf{B}\left(\mathbf{L}_1\omega +\mathbf{L}_2\omega^{(2)}\right).
\end{equation}

In the following theorem, we show that for a signal generator of the form \eqref{eq:quad_signal_gen}, the solution of the invariance equation \eqref{eq:invariance} is a quadratic manifold. The matrices that define this quadratic center manifold mapping are a natural choice for the projection matrices $\mathbf{V}_1$ and $\mathbf{V}_2$ used to construct the ROM matrices in \eqref{eq:rom_group}.

\begin{thm} \label{thm:quad_pi}
    Suppose $\bS\in\mathbb R^{r\times r}$ is chosen such that the eigenvalues of $\bA$ are disjoint from those of $\bS$ and its Kronecker sums ($\bS\oplus_k \bS$) for $k\geq 2$. Then the unique, analytic solution to the invariance equation \eqref{eq:invariance} is given by \begin{align*}
        \mathbf \Pi(\w)= \mathbf \Pi_1\w + \mathbf \Pi_2 \w^{(2)},
    \end{align*} where $\mathbf \Pi_1\in\mathbb R^{n\times r}$ and $\mathbf \Pi_2\in\mathbb{R}^{n\times r^2}$ are constant matrices given as the solution of the Sylvester equations \begin{align}
        \boldsymbol{\Pi}_1 \mathbf{S} - \mathbf{A}\boldsymbol{\Pi}_1 &= \mathbf{B}\mathbf{L}_1, \label{eq:pi1_sylvester} \\
    \boldsymbol{\Pi}_2 (\mathbf{S} \oplus \mathbf{S}) - \mathbf{A}\boldsymbol{\Pi}_2 &= \mathbf{B}\mathbf{L}_2. \label{eq:pi2_sylvester}
    \end{align}
\end{thm}

\begin{proof}
    The center manifold under the signal generator is given as the solution of the following invariance equation \begin{align*} 
    \frac{\partial \p(\w)}{\partial \omega}=\bA \Pi(\w) + \bB(\bL_1\w +\bL_2\w^{(2)}).
\end{align*}
We follow an idea similar to the derivation of Volterra series expansion for structured dynamical systems in \cite{rugh1981nonlinear} and the power series approach to solving (regulator) invariance equations in \cite{huang2004nonlinear,bai2022model}. Suppose the center manifold $\boldsymbol{\mathbf \Pi}(\omega)$ of the FOM \eqref{eq:lti_fom} under signal generated by \eqref{eq:quad_signal_gen} is given by the power series expansion (around $0$), i.e.,
\begin{equation} \label{eq:power_series_fom}
    \boldsymbol{\Pi}(\omega) = \sum_{i=1}^\infty \boldsymbol{\Pi}_i \omega^{(i)} = \boldsymbol{\Pi}_1 \omega + \boldsymbol{\Pi}_2 \omega^{(2)} + \boldsymbol{\Pi}_3 \omega^{(3)} + \dots .
\end{equation}
Substituting the power series expansion \eqref{eq:power_series_fom} into the invariance equation \eqref{eq:invariance} yields
\begin{equation}
    \sum_{i=1}^\infty \frac{\partial(\boldsymbol{\Pi}_i \omega^{(i)})}{\partial \omega} \mathbf{S}\omega = \mathbf{A} \sum_{i=1}^\infty \boldsymbol{\Pi}_i \omega^{(i)} + \mathbf{B}\mathbf{L}_1\omega + \mathbf{B}\mathbf{L}_2 \omega^{(2)}.
\end{equation}
By collecting the terms corresponding to each degree of $\omega^{(i)}$, we derive a sequence of Sylvester equations
\begin{align}
    \boldsymbol{\Pi}_1 \mathbf{S} - \mathbf{A}\boldsymbol{\Pi}_1 &= \mathbf{B}\mathbf{L}_1, \\
    \boldsymbol{\Pi}_2 (\mathbf{S} \oplus \mathbf{S}) - \mathbf{A}\boldsymbol{\Pi}_2 &= \mathbf{B}\mathbf{L}_2,  \\
    \boldsymbol{\Pi}_k (\mathbf{S} \oplus_k \bS) - \mathbf{A}\boldsymbol{\Pi}_k &= \mathbf{0}, \quad \text{for } k \geq 3, \label{eq:sylvester_k}
\end{align}
where $\oplus_k$ denotes the $k$ Kronecker sum. From the theory of Sylvester equations  (see, for example, \cite{ACA05}), for $k \geq 3$, $\boldsymbol{\Pi}_k = \mathbf{0}$ is the unique solution of the Sylvester equation if the eigenvalues of $\mathbf{A}$ are disjoint from the eigenvalues of the $k$-fold Kronecker sum of $\mathbf{S}$. Thus, by the hypothesis of the theorem, we have $\boldsymbol{\Pi}_k = \mathbf{0}$ as the unique solution to \eqref{eq:sylvester_k} for $k\geq3$. Thus, the center manifold for the FOM under the signal generator in \eqref{eq:quad_signal_gen} is given by \begin{equation*}
    \mathbf \Pi(\w)= \mathbf \Pi_1\w+\mathbf \Pi_2\w^{(2)},
\end{equation*} where $\mathbf \Pi_1\in\mathbb R^{n\times r}$ and $\mathbf \Pi_2\in\mathbb{R}^{n\times r^2}$ are constant matrices given as the solution of the Sylvester equations \eqref{eq:pi1_sylvester} and \eqref{eq:pi2_sylvester} respectively. \end{proof}

 The statement of Theorem \ref{thm:quad_pi} provides a strong motivation for the use of quadratic manifolds to approximate the FOM state since the center manifold, obtained as the solution of the nonlinear-invariance equation in \eqref{eq:invariance}, is quadratic (see Fig.~\ref{fig:quadratic_manifold}) in $\w$ and hence can be reduced to a $r$ dimensional reduced space using a linear and quadratic projection matrix. Furthermore, the theorem also provides a natural choice for the linear and quadratic projection bases in the quadratic approximation:  
\begin{equation} \label{eq:choice_V12}
    \mathbf{V}_1 := \boldsymbol{\Pi}_1 \in \mathbb{R}^{n \times r}, \quad \mathbf{V}_2 := \boldsymbol{\Pi}_2 \in \mathbb{R}^{n \times r^2}.
\end{equation} {Because these matrices are obtained by directly solving linear Sylvester equations, our approach provides an optimization-free, closed form expressions for $\bV_1$ and $\bV_2$. The QM ROM \eqref{eq:rom_state}--\eqref{eq:rom_group} constructed using $\bV_1$ and $\bV_2$ achieves nonlinear moment matching as shown in the following subsections (see Theorem~\ref{thm:moment_matching}). This construction also provides a significant practical advantage: it isolates the left projection matrix $\bW$ as a completely free design parameter, constrained only by the mild requirement that $\bW^\top\bV_1$ remains invertible. By completely decoupling the moment-matching conditions from the choice of $\bW$, our framework allows flexibility to enforce secondary numerical or physical properties (such as stability or structure preservation) without affecting the moment matching property. This can be considered analogous to the one-sided vs two-sided interpolatory projections in the linear projection-based interpolatory MOR; see, e.g.,
\cite[Thm. 3.3.1]{morAntBG20}. In the subsequent theorems, we make no additional assumptions on $\bW$ to preserve this generality. While existing literature typically relies on standard choices, such as enforcing a Galerkin-like projection ($\bW=\bV_1$) or orthogonalizing against the quadratic basis ($\bW^\top\bV_2=\mathbf{0}$), our framework automatically accommodates these choices of $\bW$.}

\begin{rem}
The assumption in Theorem~\ref{thm:quad_pi} (which carries over to Theorems~\ref{thm:identity_manifold} and~\ref{thm:moment_matching}) that the eigenvalues of $\bA$ are disjoint from those of $\bS$ and its Kronecker sums $\bS \oplus_k \bS$ has two key implications. Mathematically, it guarantees unique solutions to the Sylvester equations, ensuring that the invariant quadratic center manifold is uniquely defined. From a system-theoretic perspective, it prevents placing interpolation frequencies at the poles of the transfer function of the original LTI system, thereby ensuring well-posed steady-state dynamics.
\end{rem}

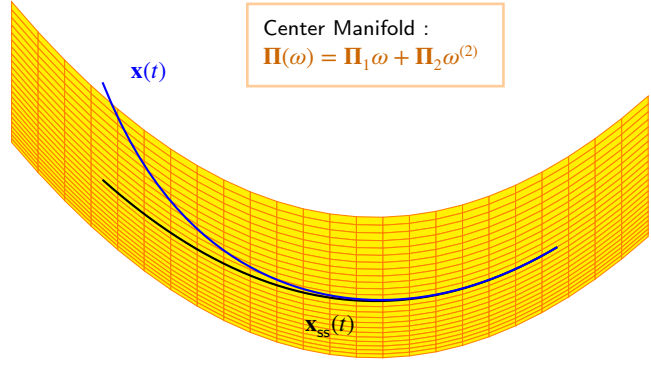
\begin{figure}
    \centering
    \begin{tikzpicture}

\begin{axis}[
        axis lines=none,
        width=10cm, height=7cm, 
        view={0}{15}, 
        clip=false 
    ]
    
    \addplot3[
        surf, 
        color=yellow,          
        faceted color=orange,  
        domain=-2:1.5, domain y=0:1
    ] {(x^2+y^2)};

    \addplot3[
        black, 
        thick, 
        samples=50, 
        domain=-1.5:1,   
        y domain=0:0,  
        samples y=1    
    ] (x,1/2,{1/4+x^2});

    \addplot3[
        blue, 
        thick, 
        samples=50, 
        domain=-1.5:1,
        y domain=0:0, 
        samples y=1   
    ] (x,1/2,{0.001*exp(1.5*(2-2*x)) + 1/4+x^2});

    \node at (axis cs: -0.25, 0.5, -0.2) [ font=\small] {\color{black} $\mathbf x_{\text{ss}}(t)$};

    \node at (axis cs: -1.1, 0.5, 4.5) [anchor=east, font=\small, blue] {$\mathbf x(t)$};

    \node at (axis cs: 0, 0.5, 5.0) [
        draw=orange!40!white, 
        fill=white, 
        inner sep=6pt, 
        align=left, 
        line width=1pt
    ] {
        \footnotesize \color{black} Center Manifold :  \\
        \small \color{orange!80!black} $ \mathbf \Pi(\omega)=\mathbf \Pi_1\omega + \mathbf \Pi_2\omega^{(2)}$ 
        \small \color{orange!80!black} 
    };

    \end{axis}
\end{tikzpicture}
    \caption{Schematic showing the quadratic center manifold of the LTI system \eqref{eq:lti_fom}. An arbitrary trajecteory $\bx(t)$ converges to its corresponding steady state $\bx_{ss}(t)$ trajectory on the center manifold.}
    \label{fig:quadratic_manifold}
\end{figure}

\subsection{Theoretical interpolation guarantees}
Based on Theorem~\ref{thm:quad_pi}, we construct the reduced system  \eqref{eq:rom_state}-\eqref{eq:rom_group} using projection matrices $\mathbf{V}_1$ and $\mathbf{V}_2$ defined in \eqref{eq:choice_V12} and $\bW$  chosen such that $\bW^\top\bV_1$ is invertible. In this section we provide the theoretical guarantees that the resulting reduced-order model (ROM) preserves the nonlinear moments of the full-order model (FOM). We first demonstrate that the center manifold of the ROM constructed using $\bV_1$ and $\bV_2$ as in \eqref{eq:choice_V12} maps to the center manifold of the FOM, thereby ensuring nonlinear moment matching along $\bW^\top$. We denote the nonlinear moments of the system as $\boldsymbol{\mathcal{M}}(\omega)$ and suppress the dependence on $\bS,\bL_1,\bL_2$ to avoid notational clutter. For the reader's convenience, we recall the equations of the ROM \eqref{eq:rom_state} interconnected with the signal generator from \eqref{eq:quad_signal_gen} given by  \begin{subequations}
\begin{align*}
    \dot{\w}(t)&=\bS\w(t),\quad \w(0)=\w_0\in\mathbb R^r\\
   \bM(\bx_r) \frac{d\mathbf{x}_r(t)}{dt} &= \mathbf{A}_r \mathbf{x}_r(t) + \mathbf{H}_r (\mathbf{x}_r(t) \otimes \mathbf{x}_r(t)) + \mathbf{B}_r \left(\bL_1\w(t)+ \bL_2(\w(t)\otimes \w(t))\right), \quad \bx_r(0)\in\mathbb R^r\\
    \mathbf{y}_r(t) &= \widehat h(\bx_r)=\mathbf{C}_r \mathbf{x}_r(t) + \mathbf{K}_r (\mathbf{x}_r(t) \otimes \mathbf{x}_r(t)), 
\end{align*}
\end{subequations} with the matrices of the ROM constructed as in \eqref{eq:rom_group} with the choice $\bV_1=\boldsymbol\Pi_1$ and $\bV_2=\boldsymbol\Pi_2$ as shown in \eqref{eq:choice_V12}. The reduced-order model (ROM) naturally has a mass-matrix term $\bM_r(\bx_r)$ multiplying the time derivative. To incorporate this into the standard nonlinear moment matching framework defined in \eqref{def:invariance_eqn}, we recast the system into a form without the mass-matrix term. By our specific choice of the projection matrix $\mathbf{W}$, the product $\mathbf{W}^\top\mathbf{V}_1$ is nonsingular. Consequently, the reduced mass matrix $\mathbf{M}_r(\mathbf{x}_r)$ is invertible at the equilibrium point $\mathbf{x}_r=0$. By continuity, there exists a neighborhood around the origin where $\mathbf{M}_r(\mathbf{x}_r)$ remains invertible. This local invertibility allows us to formulate the invariance equation for the ROM as
\begin{align}
    \frac{\partial \boldsymbol{\pi}_r}{\partial \omega} \mathbf{S}\omega &= (\mathbf{M}_r(\boldsymbol{\pi}_r(\w)))^{-1}\left(\mathbf{A}_r \boldsymbol{\pi}_r(\omega) + \mathbf{H}_r (\boldsymbol{\pi}_r(\omega) \otimes \boldsymbol{\pi}_r(\omega)) + \mathbf{B}_r (\mathbf{L}_1\omega + \mathbf{L}_2\omega^{(2)})\right) \notag \\
    \implies \mathbf{M}_r(\boldsymbol{\pi}_r(\w))\frac{\partial \boldsymbol{\pi}_r}{\partial \omega} \mathbf{S}\omega &= \mathbf{A}_r \boldsymbol{\pi}_r(\omega) + \mathbf{H}_r (\boldsymbol{\pi}_r(\omega) \otimes \boldsymbol{\pi}_r(\omega)) + \mathbf{B}_r (\mathbf{L}_1\omega + \mathbf{L}_2\omega^{(2)}). \label{eq:rom_invariance}
\end{align}
Thus, we adopt the implicit formulation in \eqref{eq:rom_invariance} as the defining invariance equation for the ROM in the subsequent results. Additionally using \eqref{eq:nonlinear_moment_def} and the definition of the reduced system in \eqref{eq:rom_state}, we have that the moment of the ROM for $\bS,\bL_1$ and $\bL_2$ is given by \begin{equation}
    \widehat{\mathcal M}(\w)=\widehat h\circ\boldsymbol \pi_r(\w)=\bC_r\boldsymbol \pi_r(\w) + \bK_r(\boldsymbol \pi_r(\w)\otimes \boldsymbol \pi_r(\w)).
\end{equation}

\begin{thm}\label{thm:identity_manifold}
Consider the ROM defined in \eqref{eq:rom_state}--\eqref{eq:rom_group} with $\mathbf{V}_1, \mathbf{V}_2$ chosen as the solution to \eqref{eq:pi1_sylvester}--\eqref{eq:pi2_sylvester} and let $\bW$ be chosen such that $\mathbf{W}^\top \mathbf{V}_1$ is invertible. Let $\boldsymbol{\pi}_r(\omega)$ be the solution to the ROM invariance equation \eqref{eq:rom_invariance}, i.e.,
\begin{equation*} \bM(\boldsymbol{\pi}_r(\w))\frac{\partial \boldsymbol{\pi}_r}{\partial \omega} \mathbf{S}\omega = \mathbf{A}_r \boldsymbol{\pi}_r(\omega) + \mathbf{H}_r (\boldsymbol{\pi}_r(\omega) \otimes \boldsymbol{\pi}_r(\omega)) + \mathbf{B}_r (\mathbf{L}_1\omega + \mathbf{L}_2\omega^{(2)})
\end{equation*} in a neighbourhood of $0$. If the eigenvalues of $\mathbf{A}_r$ are disjoint from the eigenvalues of $\mathbf{S}$ and its Kronecker sums ($\bS\oplus_k \bS$) for $k\geq 2$, then the identity mapping \begin{equation}\label{eq:identity_manifold}
    \boldsymbol{\pi}_r(\omega) = \omega
\end{equation} is the unique solution to the invariance equation \eqref{eq:rom_invariance} in a neighbourhood of $0$. 
\end{thm}

\begin{proof}
We follow a similar power series approach used in Theorem~\ref{thm:quad_pi}. Consider the invariance equation for the ROM given by
\begin{align*}\bM(\boldsymbol{\pi}_r)\frac{\partial \boldsymbol{\pi}_r}{\partial \omega} \mathbf{S}\omega = \mathbf{A}_r \boldsymbol{\pi}_r(\omega) + \mathbf{H}_r (\boldsymbol{\pi}_r(\omega) \otimes \boldsymbol{\pi}_r(\omega)) + \mathbf{B}_r (\mathbf{L}_1\omega + \mathbf{L}_2\omega^{(2)}).
\end{align*}
Consider the power series expansion for $\pr$ in a neighbourhood around 0 given as
\begin{align*}
    \pr(\w)&=\bP_1 \w+\bP_2\w^{(2)}+\bP_3\w^{(3)}+\dots = \sum_{i=1}^\infty \bP_i\w^{(i)}.
\end{align*} We will show that under the conditions stated in the theorem, $\pr(\w)=\w$ is the unique solution to the invariance equation, i.e., $\bP_1=\bI_r$ and $\bP_k=0$ for $k>1$. Using the power series expansion in the invariance equation for the ROM \eqref{eq:rom_invariance} and collecting the terms corresponding to $\w$, we have the Sylvester equation
\begin{align*}
     \bP_1 \bS &= \bA_r\bP_1 + \bB_r\bL_1\\
     \implies \bP_1\bS&=(\bW^\top\bV_1)^{-1}\bW^\top\bA\bV_1\bP_1 + (\bW^\top\bV_1)^{-1}\bW^\top\bB\bL_1.
\end{align*} 
    We have that $\bP_1=\bI_r$ is a solution to this equation since \begin{align*}
        \bW^\top (\bV_1 \bI_r\bS)&= \bW^\top (\bA\bV_1 + \bB\bL_1)
    \end{align*} from the definition of $\bV_1$ in ~\eqref{eq:pi1_sylvester}. By the hypothesis of the theorem, we have that eigenvalues of $\bA_r$ are disjoint from the eigenvalues of $\bS$ and hence this is the unique solution of the Sylvester equation. Similarly, collecting the terms corresponding to $\w^{(2)}$ in the invariance equation, we obtain
    \begin{align*}
        \bP_2(\bS\oplus \bS) +(\bW^\top \bV_1)^{-1} \bW^\top \bV_2\left((\bP_1\bS\otimes \bP_1)+(\bP_1\otimes\bP_1\bS)\right)= \bA_r\bP_2 + \bH_r(\bP_1\otimes\bP_1)+\bB_r\bL_2.
    \end{align*} Since we just established $\bP_1=\bI_r$, this simplifies to 
    \begin{align*}
        \bP_2(\bS\oplus \bS) +(\bW^\top \bV_1)^{-1} \bW^\top \bV_2\left(\bS\oplus \bS\right)= \bA_r\bP_2 + \bH_r+\bB_r\bL_2.
    \end{align*} Rearranging terms and substituting the expressions for the ROM matrices, we get
    \begin{align*}
        \bP_2(\bS\oplus \bS)-(\bW^\top \bV_1)^{-1}\bW^\top \bA \bV_1 \bP_2 & =(\bW^\top \bV_1)^{-1}\bW^\top \left(-\bV_2(\bS\oplus \bS)+  \bA\bV_2 +  \bB \bL_2\right)=0
    \end{align*} since $\bV_2$ satisfies the Sylvester equation in \eqref{eq:pi2_sylvester}. Thus, $\bP_2=0$ is the only solution of this Sylvester equation since the eigenvalues of $\mathbf{A}_r$ are disjoint from the eigenvalues of $\mathbf{S}\oplus \bS$. Similarly, collecting terms corresponding to $\w^{(k)}$ for $k\geq 3$ and substituting the expressions for $\bP_1,\dots,\bP_{k-1}$, gives us\begin{align*}
       \bP_k(\bS\oplus_k \bS) -\bA_r \bP_k=0,
    \end{align*} which has the unique solution $\bP_k=0$ under the hypothesis of the theorem. Thus, we conclude that $\pr(\w)=\w$ is the unique solution to \eqref{eq:rom_invariance} in a neighbourhood of 0.
\end{proof}

Beyond its role as a step toward proving the moment-matching result in Theorem~\ref{thm:moment_matching}, the preceding theorem provides geometric insights into the ROM center manifold. Eq.~\eqref{eq:identity_manifold} demonstrates that the ROM state directly and perfectly tracks the coordinates of the driving signal generator. This identity map simplifies the approximation of the full-order center manifold, leading directly to the following corollary. 

\begin{cor} \label{thm:along_W}
    Let $\boldsymbol{\pi_r}(\w)$ denote the solution of the ROM invariance equation \eqref{eq:rom_invariance}. Then the map $\bp(\w)$ defined as\begin{equation*}
  \bp(\w)=\bV_1\boldsymbol{\pi_r}(\w)+\bV_2(\boldsymbol{\pi_r}(\w)\otimes \boldsymbol{\pi_r}(\w))
    \end{equation*}
    solves the FOM invariance equation \eqref{eq:invariance} along direction $\bW^\top$.
\end{cor}

The significance of Corollary~\ref{thm:along_W} lies in its direct connection to projection-based model order reduction. The statement of the corollary implies that the residual of the high-dimensional invariance equation is forced to zero when projected onto the subspace defined by the left projection matrix $\bW$. This is equivalent to a generalized Petrov-Galerkin projection condition applied directly to the nonlinear manifold equations. Here, the trial space for the full-order state is defined by a quadratic manifold expansion spanned by $\bV_1$ and $\bV_2$, while the test subspace spanned by $\bW$, allows one to enforce additional desirable properties in the ROM. This provides an important link between nonlinear center manifold theory and classical projection-based frameworks. Theorems~\ref{thm:quad_pi} and~\ref{thm:identity_manifold} allow us to prove that the ROM achieves nonlinear moment matching for the signal generator $(\bS,\bL_1,\bL_2)$.

\begin{thm} \label{thm:moment_matching}
     { Suppose the assumptions of Theorems~\ref{thm:quad_pi} and~\ref{thm:identity_manifold} hold and the ROM is constructed in \eqref{eq:rom_state}--\eqref{eq:rom_group} with the choice of $\bV_1,\bV_2$ as in \eqref{eq:choice_V12} and $\bW$ chosen so that $\bW^\top\bV_1$ is invertible. Then, the center manifold $\boldsymbol{\pi_r}(\w)$ of the ROM maps to the center manifold $\p(\w)$ of the LTI system in \eqref{eq:lti_fom} under the map} \begin{equation*}
        \p(\w)=\bV_1\boldsymbol{\pi_r}(\w)+\bV_2(\boldsymbol{\pi_r}(\w)\otimes \boldsymbol{\pi_r}(\w)).
    \end{equation*} Moreover, the nonlinear moment, $ \widehat{\mathcal M}(\cdot)$, of the ROM defined in \eqref{eq:rom_state}--\eqref{eq:rom_group} matches the nonlinear moment, $\mathcal M(\cdot)$, of the LTI system  in \eqref{eq:lti_fom}, i.e., \begin{align*}
        \mathcal M(\w)=\widehat {\mathcal M}(\w),
    \end{align*} for all $\w$ in a neighborhood of $\w=0$.
\end{thm}

\begin{proof}
    The nonlinear moment for the ROM is given as, \begin{align*}
        \widehat{\mathcal M}(\w)=\bC_r\boldsymbol \pi_r(\w) + \bK_r(\boldsymbol \pi_r(\w) \otimes \boldsymbol \pi_r(\w)).
    \end{align*}
    Since $\p(\w)=\bV_1\w+\bV_2\w^{(2)}$ from Theorem~\ref{thm:quad_pi} and $\boldsymbol{\pi_r}(\w)=\w$ in a neighbourhood around 0 from 
    Theorem~\ref{thm:identity_manifold}, we have that\begin{align*}
        \mathcal{M} (\w)=\bC\left(\bV_1\w+\bV_2\w^{(2)}\right)=\bC_r\boldsymbol{\pi_r}(\w)+\bK_r\boldsymbol{\pi_r}^{(2)}(\w)= \widehat{\mathcal{M}}(\w)
    \end{align*}for all $\w$ in a neighborhood of $0$. 
\end{proof} 

This theorem shows that the choice of $\bV_1$ and $\bV_2$ made in \eqref{eq:choice_V12} achieves nonlinear moment matching in the sense of \cite{astolfi2010}. We also note that this result hold for any choice of $\bW$ such that $\bW^\top\bV_1$ is invertible. Corollary~\ref{thm:along_W} provides some insight into the role of $\bW$ in the moment matching framework since we can interpret $\bW$ to be the directions along which we solve the FOM invariance equations. Because this choice of $\bW$ is arbitrary, we have freedom in enforcing additional properties in our ROM. One such condition is forcing the ROM to have no mass matrix term, i.e., pick $\bW$ such that $\bW^\top\bV_2=0$ so that $\bM_r(\bx_r)=\bI_r$. Combining the theoretical results established in this section and the reduced system we derived in Section~\ref{sec:qm_rom}, we propose an algorithm to compute the quadratic reduced system \eqref{eq:rom_state}--\eqref{eq:rom_group}. 
\begin{algorithm} 
\DontPrintSemicolon
\caption{Moment matching quadratic manifold approach}
\label{alg:moment_matching_rom}

\KwIn{Linear Time-Invariant (LTI) system matrices $(\bA, \bB, \bC)$ as given in \eqref{eq:lti_fom}.}
\KwOut{Reduced Order Model (ROM) that achieves moment matching as shown in Theorem~\ref{thm:moment_matching}}\vspace{2mm}
\begin{enumerate}
    \item Choose matrices $\bS\in\mathbb R^{r\times r}$, $\bL_1\in\mathbb{R}^{m\times r}$, and $\bL_2\in\mathbb{R}^{m\times r^2}$ to interpolate desired input signals generated by the signal generator. Some canonical choices for the signal generator are discussed in Section~\ref{sec:signal_space}.
    
    \item Solve Sylvester equations for $\bV_1\in\mathbb R^{n\times r}$ and $\bV_2\in\mathbb R^{n\times r^2}$:
    \begin{align*}
        \bA\bV_1 + \bB\bL_1 &= \bV_1 \bS \\
        \bA\bV_2 + \bB\bL_2 &= \bV_2 (\bS \oplus \bS)
    \end{align*} 
    These Sylvester equations define the center manifold of the full order system as shown in Theorem~\ref{thm:quad_pi}. 

    \item Choose a projection matrix $\bW\in\mathbb R^{n\times r}$ such that $\bW^\top \bV_1$ is invertible. Some choices include:
    \begin{itemize}
        \item Enforce $\bW^\top \bV_2 = \mathbf{0}$ (Eliminates state-dependent mass-matrix term) 
        \item Set $\bW = \bV_1$ (Galerkin projection)
    \end{itemize}

    \item Compute the reduced-order model matrices:
    \begin{align*}
        \bA_r &= (\bW^\top \bV_1)^{-1} \bW^\top \bA \bV_1,  \quad \bB_r = (\bW^\top \bV_1)^{-1} \bW^\top \bB, \quad   \bC_r = \bC \bV_1, \quad   \bK_r = \bC \bV_2,\\
        \bH_r &= (\bW^\top \bV_1)^{-1} \bW^\top \bA \bV_2,\quad   \bE_r = (\bW^\top \bV_1)^{-1} \bW^\top \bV_2,  \quad   \bM_r(\bx_r) = \bI_r + \bE_r (\bx_r \otimes \bx_r)
    \end{align*}

    \item Construct the interpolating QM ROM given by:
    \begin{align*}
        \bM_r(\bx_r) \, \dot{\bx}_r(t) &= \bA_r \bx_r(t) + \bB_r \bu(t) + \bH_r \big(\bx_r(t) \otimes \bx_r(t)\big) \\
        \by_r(t) &= \bC_r \bx_r(t) + \bK_r (\bx_r(t)\otimes \bx_r(t))
    \end{align*}

    \textbf{Theoretical Guarantee:} By construction, the resulting QM ROM matches the nonlinear moment of the full-order model, i.e., $\hat{\mathcal{M}}(\omega) = \mathcal{M}(\omega)$. This ensures exact steady-state dynamic reproduction for any input signal generated by $(\bS, \bL_1, \bL_2)$ (Theorem~\ref{thm:moment_matching}).
\end{enumerate}

\end{algorithm}

Algorithm \ref{alg:moment_matching_rom} summarizes the step-by-step procedure for constructing the optimization-free quadratic reduced-order model. We begin, in Step 1, by defining the signal generator matrices $\bS, \bL_1$, and $\bL_2$, which encode the targeted interpolation frequencies and tangential directions. Then, in Step 2, these parameters are used to determine the linear and quadratic basis matrices, $\bV_1$ and $\bV_2$, by solving the two decoupled linear Sylvester equations \eqref{eq:pi1_sylvester} and \eqref{eq:pi2_sylvester}. Step 3 chooses the left projection matrix $\bW$, which can be tailored to enforce specific geometric properties (such as a standard Galerkin projection where $\bW=\bV_1$ or an orthogonality condition $\bW^\top\bV_2=0$), provided that the invertibility condition on $\bW^\top\bV_1$ is satisfied. Using these bases, in Step 4, one projects the full-order system to efficiently compute the reduced matrices, which notably includes assembling the terms for the state-dependent mass matrix $\bM(\bx_r)$. Finally, in Step 5, the resulting nonlinear reduced-order model is computed, which is theoretically guaranteed to match the targeted nonlinear moments of the original full-order system as per Theorem~\ref{thm:moment_matching}.

\begin{rem}
{The online storage and evaluation of the quadratic ROM constructed via Algorithm~\ref{alg:moment_matching_rom} are independent of the full-order dimension $n$. Although the projection matrices $\bV_1\in\mathbb R^{n\times r}$ and $\bV_2\in\mathbb R^{n\times r^2}$ are computed during the offline phase, they are never loaded during online execution. By precomputing all reduced-order operators offline, the online memory complexity scales strictly as $\mathcal O(r^3)$.}
\end{rem}

\section{Enforcing application specific interpolation conditions}
\label{sec:signal_space}

The selection of an appropriate signal generator matrix $\mathbf{S}$ for moment matching is inherently problem-dependent, as it dictates the specific family of trajectories the reduced-order model (ROM) is constrained to replicate. In this section, we detail canonical choices for $\mathbf{S}$ designed to capture distinct dynamic behaviors, ranging from persistent steady-state oscillations to quasi-polynomial transient signals, and provide explicit analytical solutions for the resulting manifold templates. 

\subsection{Real canonical forms for sinusoidal inputs}\label{subsec:sin_inputs}

To interpolate purely sinusoidal steady-state signals at a given set of driving frequencies, the signal generator matrix must possess purely imaginary eigenvalues $\{\pm j\mu_1, \pm j\mu_2, \dots, \pm j\mu_k\}$. To ensure that all subsequent reduction stages, including the derivation of the projection matrices $\mathbf{V}_1$ and $\mathbf{V}_2$, and consequently the ROM operators themselves, remain strictly within the real domain, we enforce a real canonical form  for $\mathbf{S}$~\cite{astolfi2010, morAntBG20}:
\begin{equation}\label{eq:S_sinusoidal}
    \mathbf{S} = \operatorname{blkdiag}\left(
    \begin{bmatrix} 0 & \mu_1 \\ -\mu_1 & 0 \end{bmatrix}, \, \dots, \, 
    \begin{bmatrix} 0 & \mu_k \\ -\mu_k & 0 \end{bmatrix}
    \right) \in \mathbb{R}^{r \times r},
\end{equation}
where $r = 2k$. By matching the moment with this block-diagonal structure, the ROM accurately replicates the frequency response of the original high-dimensional system at the selected frequencies. {This choice of $\bS$ results in inputs of the form \begin{align}\label{eq:sin_inputs}
    \bu(t)= \left(\sum_{i=1}^r c_{i,1}\sin(\mu_i t) + c_{i,2}\cos(\mu_i t)\right) + \left(\sum_{i,j=1}^r c_{ij,1} \sin((\mu_i+\mu_j)t) + c_{ij,2} \cos((\mu_i+\mu_j)t)\right),
\end{align} where $c_{i,1},c_{i,2},c_{ij,1}, c_{ij,2}$ are constants dependent on $\bL_1,\bL_2,\w_0$ for $1\leq i,j,\leq r$. The constructed ROM, thereby, matches the steady-state behaviour of the FOM for all inputs of the form in \eqref{eq:sin_inputs}. We note that the ROM not only interpolates frequencies $\mu_i,~1\leq i\leq r$, but also the mixed interpolation frequencies $\mu_i+\mu_j, 1\leq i,j\leq r$. Furthermore, this real canonical representation of $\bS$ eliminates the need for complex arithmetic, directly yielding a  real-valued reduced-order system. This construction is conceptually analogous to the basis transformations employed to enforce realness within interpolatory model reduction and Loewner frameworks \cite{morAntBG20}.

\subsection{Closed-form solutions for diagonal signal generators} \label{subsec:diagonal_gen}

To illustrate the explicit algebraic bridge between center manifold theory and classical transfer function interpolation, consider a diagonalized signal generator configuration where $\mathbf{S} = \operatorname{diag}(s_1, \dots, s_r) \in \mathbb{C}^{r \times r}$. Let $\mathbf{\ell}_i \in \mathbb{R}^{m}$ denote the $i$-th column of $\mathbf{L}_1$ and let $\mathbf{b}_{ij} \in \mathbb{R}^{m}$ represent the column of $\mathbf{L}_2$ associated with the interacting state components $\omega_i \omega_j$ (for $i,j = 1, \dots, r$). Under this diagonal framework, the Sylvester equations defining the linear and quadratic manifold mappings decouple column-by-column. This yields the following exact, closed-form expressions for the columns of $\boldsymbol{\Pi}_1$ and $\boldsymbol{\Pi}_2$:
\begin{align}
    \boldsymbol{\Pi}_1 &= \left[ (s_1\mathbf{I} - \mathbf{A})^{-1}\mathbf{B}\mathbf{\ell}_1, \, \dots, \, (s_r\mathbf{I} - \mathbf{A})^{-1}\mathbf{B}\mathbf{\ell}_r \right] \in\mathbb R^{n\times r},\label{eq:pi1_closed} \\
    \boldsymbol{\Pi}_2 &= \left[ \dots, \, \left((s_i + s_j)\mathbf{I} - \mathbf{A}\right)^{-1}\mathbf{B}\mathbf{b}_{ij}, \, \dots \right]\in\mathbb R^{n\times r^2}. \label{eq:pi2_closed}
\end{align}
Recalling the FOM transfer function $\mathbf{H}(s) = \mathbf{C}(s\mathbf{I} - \mathbf{A})^{-1}\mathbf{B}$, the steady-state nonlinear system output profile $\mathcal{M}(\omega) = \mathbf{C}\boldsymbol{\Pi}_1\omega + \mathbf{C}\boldsymbol{\Pi}_2(\omega \otimes \omega)$ can be mapped directly to frequency-domain evaluations of the full-order model as 
\begin{equation}\label{eq:moment_explicit}
    \mathcal{M}(\omega) = \sum_{i=1}^{r} \mathbf{H}(s_i)\mathbf{\ell}_i \omega_i + \sum_{i=1}^{r}\sum_{j=1}^{r} \mathbf{H}(s_i + s_j)\mathbf{b}_{ij} \omega_i \omega_j.
\end{equation}
Equation \eqref{eq:moment_explicit} provides a clear interpretation of the framework: the linear manifold component interpolates the system's response at the fundamental driving frequencies $s_i$, while the quadratic component captures the "mixed" frequencies $s_i + s_j$. Thus, the reduced-order model incorporates the mathematical effects of $r$ frequencies as well as $(r^2+r)/2$ mixed frequencies while maintaining a strictly $r$-dimensional reduced state-space. 

\begin{rem}\label{rem:rom_dimension}
For readers acquainted with the projection-based interpolatory model reduction for matching linear moments, the structure of $\boldsymbol{\Pi}_1$ and $\boldsymbol{\Pi}_2$ would look familiar. However,
the structural formulation in \eqref{eq:moment_explicit} reveals a crucial distinction regarding the interpolation properties of the reduced-order model. Under the classical linear rational interpolation paradigm \cite{morAntBG20}, matching the full-order transfer function at the fundamental frequencies $s_i$ and the mixed harmonic frequencies $s_i + s_j$ would require a linear ROM of total dimension $r+(r+r^2)/2=\mathcal{O}(r^2)$. A linear subspace treats the mixed-frequency modes $s_i + s_j$ as independent, uncoupled coordinates. While such a high-dimensional linear system achieves identical one-sided interpolation points, it is not guaranteed to preserve the system's intrinsic center manifold geometry.  
\end{rem}

\subsection{Matching higher-order derivatives}
The proposed framework can be systematically extended to capture complex transient dynamics such as quasipolynomial signals that naturally arise in systems characterized by repeated eigenvalues or damped modes. This is accomplished by constructing a signal space that interpolates the higher-order derivatives of the transfer function. To illustrate this construction, consider a signal generator defined by the matrices
\begin{equation}\label{eq:S_quasipolynomial}
    \mathbf{S} = \begin{bmatrix}
        \mu & 1\\
        0 & \mu
    \end{bmatrix} \in \mathbb{C}^{2 \times 2}, \quad 
    \mathbf{L}_1 = \begin{bmatrix} \ell_1 & \ell_2 \end{bmatrix} \in \mathbb{C}^{m\times 2}, \quad 
    \mathbf{L}_2 = \begin{bmatrix} b_1 & b_2 & b_3 & b_4 \end{bmatrix} \in \mathbb{C}^{m\times 4}.
\end{equation}

The resulting nonlinear moment $\mathcal{M}(\mathbf{w})$ can then be evaluated as
\begin{align}
    \mathcal{M}(\mathbf{w}) &= \bC \mathbf{V}_1 \mathbf{w} + \bC \mathbf{V}_2 (\mathbf{w} \otimes \mathbf{w}) \nonumber\\
    &= \mathbf{H}(\mu)(\ell_1 w_1 + \ell_2 w_2) + \mathbf{H}(2\mu)(b_1 w_1^2 + b_2 w_1 w_2 + b_3 w_2 w_1 + b_4 w_2^2) \nonumber\\
    &\quad + \mathbf{H}'(\mu)\ell_1 w_2 + \mathbf{H}'(2\mu)\left(b_1 w_1 w_2 + b_1 w_2 w_1 + (b_2+b_3)w_2^2\right) + \mathbf{H}''(2\mu)b_1 w_2^2 \nonumber\\
    &= c_1 \mathbf{H}(\mu) + c_2 \mathbf{H}(2\mu) + c_3 \mathbf{H}'(\mu) + c_4 \mathbf{H}'(2\mu) + c_5 \mathbf{H}''(2\mu).
\end{align}

This formulation yields a basis of amplitude-modulated, exponentially decaying or oscillating signals. Consequently, any scalar control input $\mathbf{u}(t)$ generated within this signal space takes the explicit form
\begin{equation}\label{eq:u_transient}
    \mathbf{u}(t) = c_1 e^{\mu t} + c_2 t e^{\mu t} + c_3 e^{2\mu t} + c_4 t e^{2\mu t} + c_5 t^2 e^{2\mu t},
\end{equation}
where the weighting coefficients $c_1, \dots, c_5$ are uniquely determined by the initial conditions $\mathbf{w}(0)$ and the matrices $\mathbf{L}_1, \mathbf{L}_2$. Because $s \in \mathbb{C}$, the inputs and projection matrices are, in general, complex-valued. To enforce real-valued physical quantities, $\mathbf{S}$ can be represented in real canonical form to interpolate $s = i\mu$ ($\mu \in \mathbb{R}$):
\begin{equation}\label{eq:S_real}
    \mathbf{S} = \begin{bmatrix}
        0 & \mu & 1 & 0\\
        -\mu & 0 & 0 & 1\\
        0 & 0 & 0 & \mu\\
        0 & 0 & -\mu & 0
    \end{bmatrix} \in \mathbb{R}^{4 \times 4}.
\end{equation} 
By parameterizing the signal generator in this real-valued canonical structure with a Jordan block, the nonlinear moment matching condition reduces directly into solving matrix Sylvester equations. This  aligns with the foundational Sylvester equation formulations for projection based model reduction for LTI system in \cite{gallivan2004sylvester}.

\section{Numerical results}\label{sec:numerical}

While the theoretical results and the reduction framework proposed in this work are broadly applicable to any linear time-invariant (LTI) system of the form \eqref{eq:lti_fom} including regimes where conventional linear subspaces are proven to perform effectively, linear projection methods remain fundamentally ill-suited for efficiently capturing shift-variant dynamics or LTI systems with slow decay of Hankel singular values (i.e., slow Kolmogorov $n$-width decay).  Thus, we consider classical transport-dominated problems, specifically the one-dimensional advection-diffusion equation and the one-dimensional wave equation, with their characteristic slow decay of the (normalized) Hankel singular values shown in Figure~\ref{fig:combined_hsv}. Consequently, such systems serve as ideal benchmarks to demonstrate the efficacy and computational gains of the proposed quadratic manifold framework. All numerical simulations were implemented on \matlab R2023b (version 23.2.0, Update 4) on a laptop equipped with an Apple M3 Pro chip, running macOS 14.6.1. The repository to generate all plots and simulations in this paper can be found in \cite{padhi_code}.

\begin{figure}
    \centering
    \begin{subfigure}{0.45\linewidth}
        \centering
        \includegraphics[width=\linewidth]{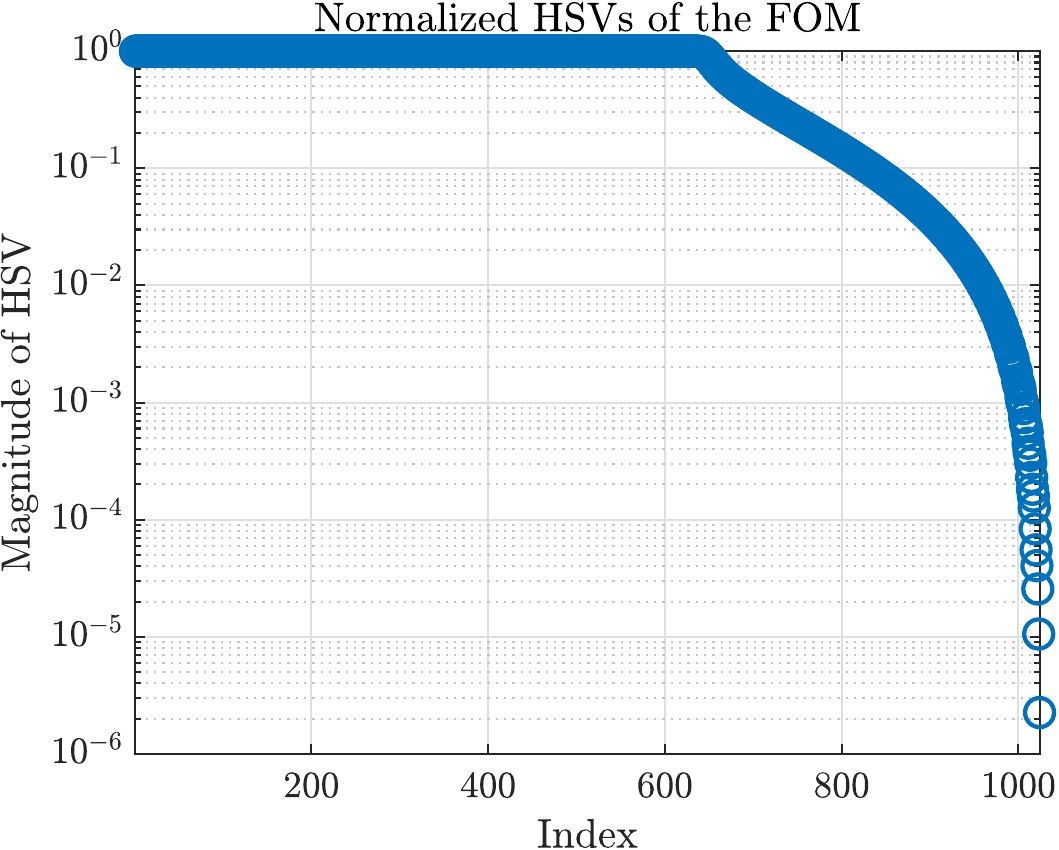}
        \caption{One-dimensional advection equation ($n=1024$).}
        \label{fig:hsv}
    \end{subfigure}
    \hfill
    \begin{subfigure}{0.45\linewidth}
        \centering
        \includegraphics[width=\linewidth]{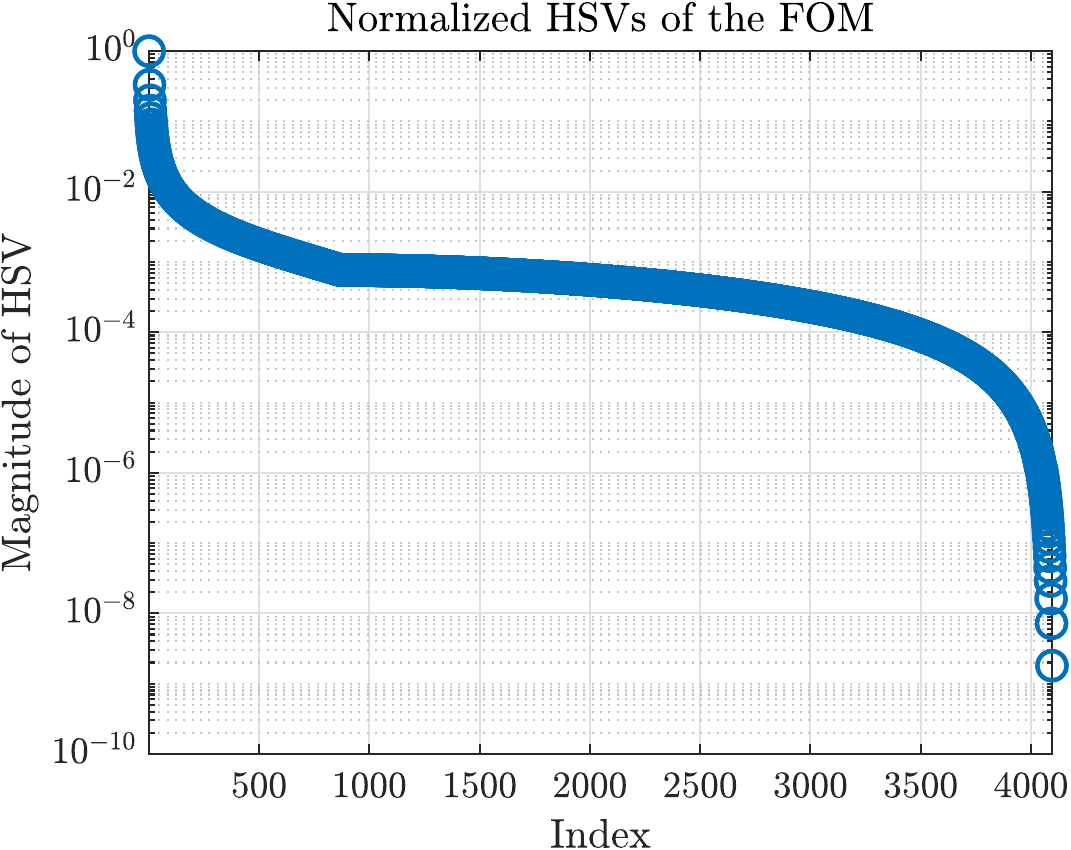}
        \caption{One-dimensional damped wave equation ($n=4096$).}
        \label{fig:hsv_decay}
    \end{subfigure}
    \caption{Decay behavior of the normalized Hankel singular values (HSV), illustrating the characteristically slow decay rates for the respective systems with the chosen parameters.}
    \label{fig:combined_hsv}
\end{figure}

\subsection{One-dimensional linear transport equation} \label{sec:adv_eqn}
As the first benchmark, we consider the dynamics of a one-dimensional linear transport (advection) equation defined on a spatial domain of length 1. This problem models the pure translation of a scalar field with a constant velocity $v=1$ directed toward the right. The system is governed by the following first-order hyperbolic partial differential equation (PDE):
\begin{equation}
\frac{\partial \xi}{\partial t} +  \frac{\partial \xi}{\partial z} = 0, \quad z \in (0, 1), \quad t > 0.
\label{eq:advection_pde}
\end{equation}
Here, $\xi(z,t)$ represents the continuous scalar state (for example displacement) at spatial position $z$ and time $t$. The system is driven by a time-dependent boundary condition on the left ($z = 0$), given by
\begin{equation}
\label{eq:boundary_condition}
\xi(0, t) = \bu(t),
\end{equation}
where $\bu(t)$ is a scalar input. The system is observed at the right boundary ($z = 1$), with $\by(t) = \xi(1, t)$ representing the scalar output of the system. By explicitly enforcing the Dirichlet boundary condition at the inlet node ($z=0$), the remaining $n = N$ interior and outlet nodes form the active state vector $\bx(t) \in \mathbb{R}^n$. The resulting LTI system is in the form of~\eqref{eq:lti_fom} where the state-space matrices $\mathbf{A} \in \mathbb{R}^{n \times n}$ and $\mathbf{B} \in \mathbb{R}^{n \times 1}$ are obtained using Chebyshev pseudospectral method for spatial discretization. This ensures we do not introduce numerical diffusion to the system, thereby making the system diffusion-dominated. The output matrix $\mathbf{C} \in \mathbb{R}^{1 \times n}$ acts as a selector vector isolating the node corresponding to the observation node at $z=1$ and we set $n=1024$ to obtain a LTI system of that dimension. As we see in Figure~\ref{fig:hsv}, the system showcases extreme slow decay in the singular values making it a challenging problem for model reduction.

\subsubsection*{Constructing the interpolating reduced system}

We construct a reduced-order system as defined in Section~\ref{sec:qm_rom} that achieves moment matching for a specified signal space $(\bS, \bL_1, \bL_2)$ as shown in Theorem.~\ref{thm:moment_matching}. The signal generator matrices are chosen as,

\begin{equation}\label{eq:signal_advection}
    \begin{aligned}
    \bS &= \operatorname{blkdiag}\left(\left[\begin{array}{cc}
       0&\mu_1  \\
     \overline{\mu}_1 & 0
    \end{array}\right],  \dots, \left[\begin{array}{cc}
       0& \mu_{r/2}   \\
      \overline{\mu}_{r/2}& 0  
    \end{array}\right]\right)\in\mathbb R^{r\times r}, \\
    \bL_1 &= [1~1~\dots~1] \in \mathbb{R}^{1 \times r}, \quad
    \bL_2 = [1~1~\dots~1] \in \mathbb{R}^{1 \times r^2}.
\end{aligned}
\end{equation}

We set the reduction order to $r=20$. The interpolation frequencies are chosen as complex conjugate pairs, spaced logarithmically along the imaginary axis between $10^{1}$ and $10^{3}$. This selection also allows us to study dynamics with different timescales using slow varying and fast varying input signals. For simplicity and uniformity, the tangential weights are set to ones. Under this choice of signal generator, we construct a quadratic ROM of the form \eqref{eq:rom_state}--\eqref{eq:rom_group} using Algorithm~\ref{alg:moment_matching_rom}.

\begin{figure}
    \centering
    \includegraphics[width=1\linewidth]{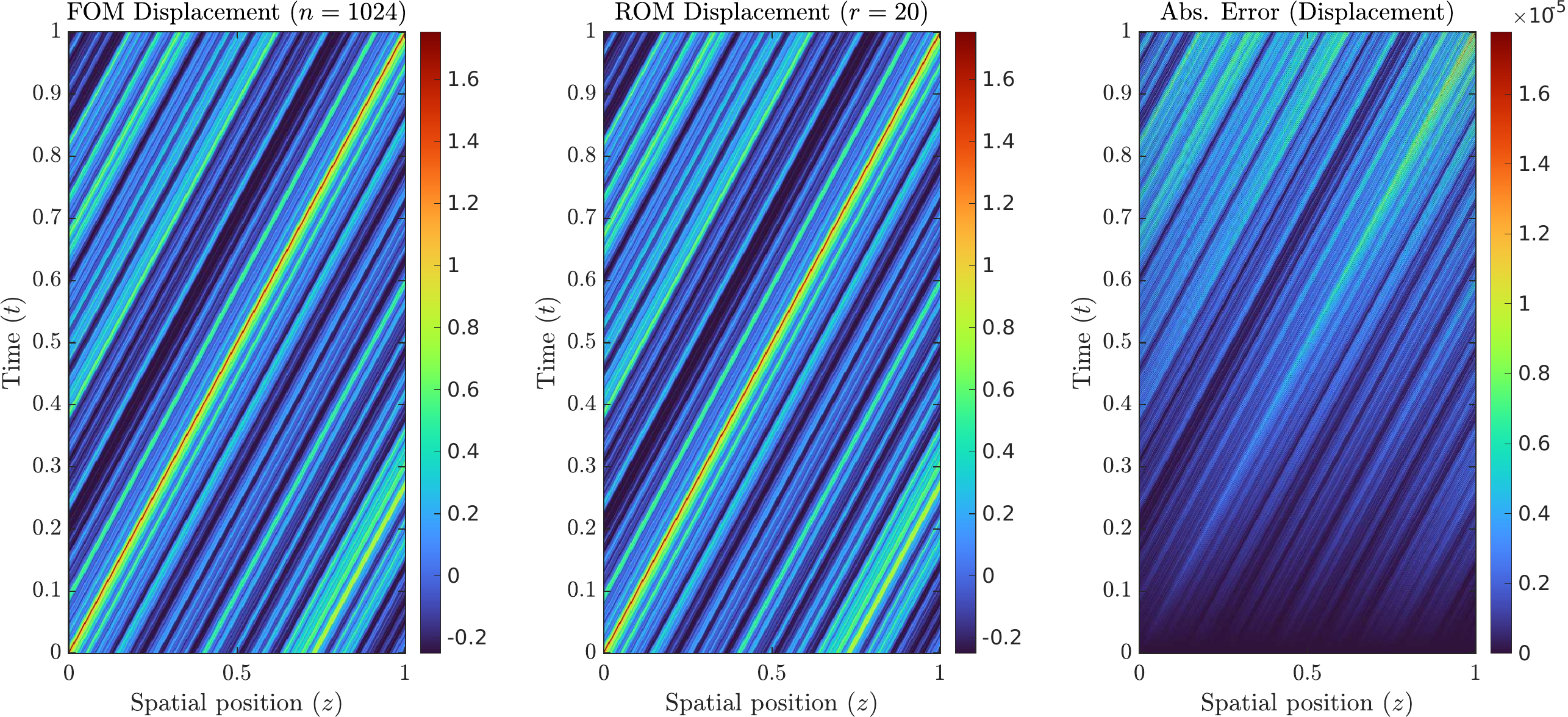}
    \caption{Steady state reconstruction of the discretized linear transport equation system on the center manifold. }\label{fig:state_reconstruction}
\end{figure}

 To show that the quadratic ROM \eqref{eq:rom_state}-- \eqref{eq:rom_group} constructed from the algorithm indeed achieves moment matching, we generate a signal from the signal generator in \eqref{eq:signal_advection} with a random initial condition $\omega_0\in\mathbb R^{20}$. This input signal is then used to simulate the QM ROM and FOM from $t\in[0,1]$. We plot the state evolution and output of both the full and reduced order systems assuming that the two systems start on their respective center manifolds, i.e., \begin{align*}
    \bx(0)&=\boldsymbol{\pi}(\w_0)=\bV_1\w_0 + \bV_2 \w_0^{(2)}\in \mathbb R^{1024},\quad \quad\bx_r(0)=\boldsymbol{\pi}_r(\w_0)=\w_0\in\mathbb R^{20}.
\end{align*} We then lift the reduced states $\bx_r(t)\in\mathbb R^{20}$ (for $t>0$) to the full space using the quadratic mapping, 
\begin{align*}
\boldsymbol{\Phi}(\mathbf{x}_r(t)) = \mathbf{V}_1 \mathbf{x}_r(t) + \mathbf{V}_2 (\mathbf{x}_r(t) \otimes \mathbf{x}_r(t)) \in \mathbb{R}^{1024},
\end{align*}
allowing us to evaluate the state reconstruction error against the true FOM solution. Figure~\ref{fig:state_reconstruction} illustrates the space-time evolution of the state for the full-order model (FOM), the proposed quadratic manifold reduced-order model (ROM), and the resulting absolute reconstruction error under the input signal generated by the signal generator in \eqref{eq:signal_advection}. The distinct diagonal bands in the left and middle panels prominently capture the continuous wave propagation of the FOM ($n = 1024$) and the ROM ($r = 20$), respectively, across the spatial grid over the simulation time $t \in [0, 1]$. Despite a drastic reduction in dimensionality, the quadratic manifold ROM successfully tracks the travelling wavefronts accurately. 

\begin{figure}
    \centering
    \includegraphics[width=0.6\linewidth]{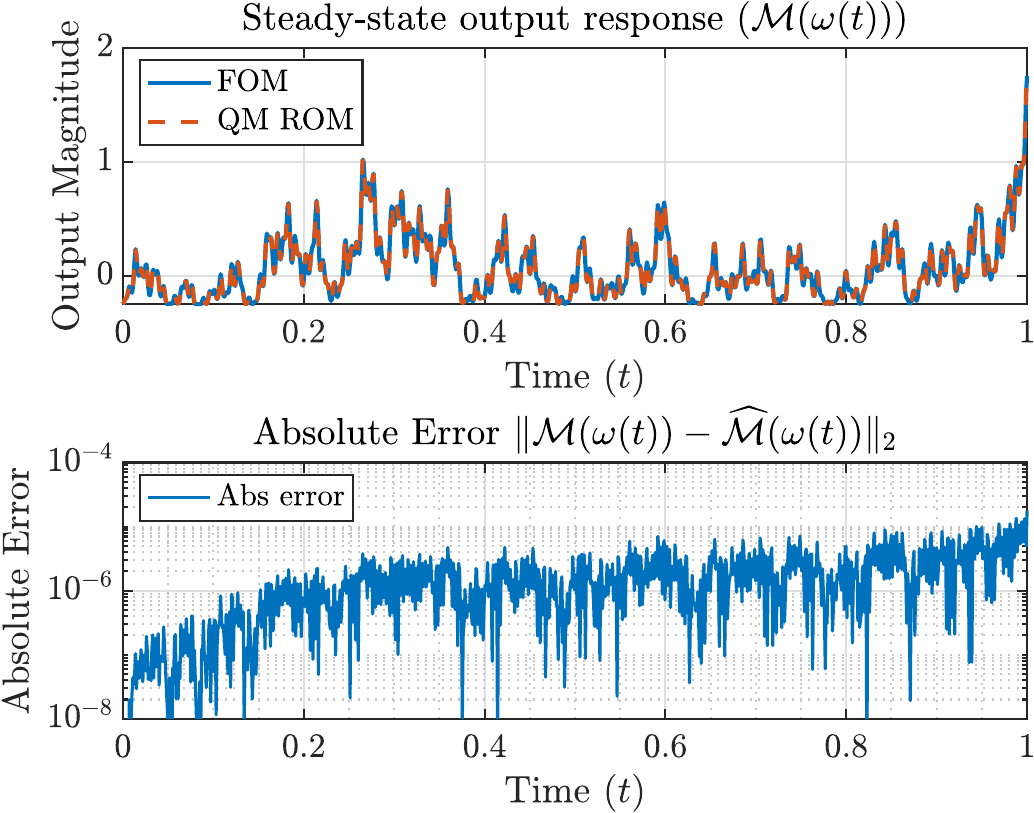}
    \caption{Numerical verification of moment matching property for the one-dimensional advection equation}
    \label{fig:moment_advection}
\end{figure}

We recall that the steady state output of the FOM under the signal generator is given by the mapping $\mathcal M(\w(t))$ where $\w(t)$ represents the signal state vector. The top panel in Figure~\ref{fig:moment_advection} displays the time history of the steady-state system outputs $\mathcal{M}(\omega(t))$ and $\widehat{\mathcal M}(\w(t))$ for the FOM and the quadratic manifold ROM (QM ROM), respectively, over $t \in [0, 1]$. The trajectories of the full and reduced models are virtually indistinguishable, with the red dashed line perfectly tracking the complex, multi-frequency oscillations of the reference full-order system. The bottom panel plots the absolute output error $|\mathcal{M}(\omega(t)) - \widehat{\mathcal{M}}(\omega(t))|$ for both system components on a semi-logarithmic scale. Theoretically, the quadratic ROM tracks the steady state outputs and center manifold of the full order system exactly. However, minor discrepancies accumulate in practice due to numerical integration. Nevertheless, the relative error remains strictly bounded and low throughout the entire simulation horizon. These results showcase the capability of the proposed framework to accurately capture transport-dominated dynamics within an exceptionally low-dimensional state space. Table~\ref{tab:runtime_advection} reports the empirical online runtimes comparing the QM-ROM against the FOM, with the QM-ROM offering over a $380\times$ online speedup. Further computational profiling and theoretical time-complexity details are provided in Section~\ref{sec:computational_profiling}.

\subsection{One dimensional damped wave equation}
As the next benchmark, we consider the dynamics of a one-dimensional damped wave equation on a spatial domain of length $1$. The system is governed by the PDE given below.
\begin{equation}
    \frac{\partial^2 \xi}{\partial t^2} + \gamma \frac{\partial \xi}{\partial t} = c^2 \frac{\partial^2 \xi}{\partial z^2} + f(z,t),
\end{equation}
where $\xi(z,t)$ is the displacement at spatial position $z$ and time $t$, $c$ represents the speed of wave propagation, $\gamma>0$ is the viscous damping coefficient, and $f(z,t)$ represents the external applied force. The PDE is spatially discretized over $n$ internal grid points using a standard central finite difference method. Let $\mathbf{w}(t) \in \mathbb{R}^n$ represent the vector of displacements at these grid points. To cast this second-order system into a standard first-order form, we augment the displacement and velocity vectors to define the state vector as $\bx(t) = \begin{bmatrix} \mathbf{w}(t)^T & \dot{\mathbf{w}}(t)^T \end{bmatrix}^T \in \mathbb{R}^{2n}$. This yields a continuous-time linear time-invariant (LTI) system of dimension $2n$ in the form of~\eqref{eq:lti_fom} where $\bu(t)$ is the scalar input and $\by(t)$ is the scalar output (the measured displacement). The spatial distribution of the input, originally $f(z,t)$ in the PDE, is captured by the input matrix $\bB$. The system matrix $\bA \in \mathbb{R}^{2n \times 2n}$, the input matrix $\bB \in \mathbb{R}^{2n \times 1}$, and the output matrix $\bC \in \mathbb{R}^{1 \times 2n}$ are defined as
\begin{equation}
    \bA = \begin{bmatrix} \mathbf{0}_{n \times n} & \bI_{n \times n} \\ c^2 \bL_n & -\gamma \bI_{n \times n} \end{bmatrix}, \quad \bB = \begin{bmatrix} \mathbf{0}_{n \times 1} \\ \tilde{\bB} \end{bmatrix}, \quad \bC = \begin{bmatrix} \tilde{\bC} & \mathbf{0}_{1 \times n} \end{bmatrix},
\end{equation}
where $\bL_n$ is the discrete Laplacian matrix representing the finite difference approximation of the spatial derivative. In this specific implementation, the input force is applied at the center of the domain. Therefore, $\tilde{\bB}$ places a value of $1$ at the index corresponding to the middle node and zeros everywhere else, meaning the input directly affects only the velocity of that specific node. Similarly, the measurement matrix $\tilde{\bC}$ is designed to extract the displacement of this exact same middle node.

For the numerical experiments, the physical parameters are chosen as $c=1$, and $\gamma=10^{-8}$. Spatially discretizing the domain with $n=2048$ internal grid points yields an LTI system of dimension $2n = 4096$. The selection of a small, non-zero damping coefficient $\gamma$ ensures that the system dynamics remain transport-dominated, while strictly displacing the eigenvalues of $\bA$ from the imaginary axis into the open left-half plane, a standard and realistic assumption when numerically simulating such systems. The Hankel singular values (HSVs) for this system are shown in Figure~\ref{fig:hsv_decay}. As illustrated, the HSVs exhibit an exceptionally slow decay rate. This characteristic is typical for wave propagation problems and indicates that traditional linear model reduction methods would require a high-dimensional state to capture the dynamics accurately. 

\subsubsection*{Constructing the interpolating reduced system}

In this section, we construct a reduced-order system that achieves moment matching for a specified signal space defined by the tuple $(\bS, \bL_1, \bL_2)$. The signal generator matrix $\bS$ is selected to match the imaginary frequencies close to the eigenvalues of the original system matrix $\bA$. For simplicity and uniformity, the tangential weights are set to ones. The resulting signal generator matrices are chosen as,
\begin{align*}
    \bS &= \operatorname{blkdiag}\left(\left[\begin{array}{cc}
       0&\mu_1  \\
     \overline{\mu}_1 & 0
    \end{array}\right],  \dots, \left[\begin{array}{cc}
       0& \mu_{r/2}   \\
      \overline{\mu}_{r/2}& 0  
    \end{array}\right]\right)\in\mathbb R^{r\times r}, \\
    \bL_1 &= [1~1~\dots~1] \in \mathbb{R}^{1 \times r}, \quad\quad\bL_2 = [1~1~\dots~1] \in \mathbb{R}^{1 \times r^2}.
\end{align*}

\begin{figure}
    \centering
    \includegraphics[width=\linewidth]{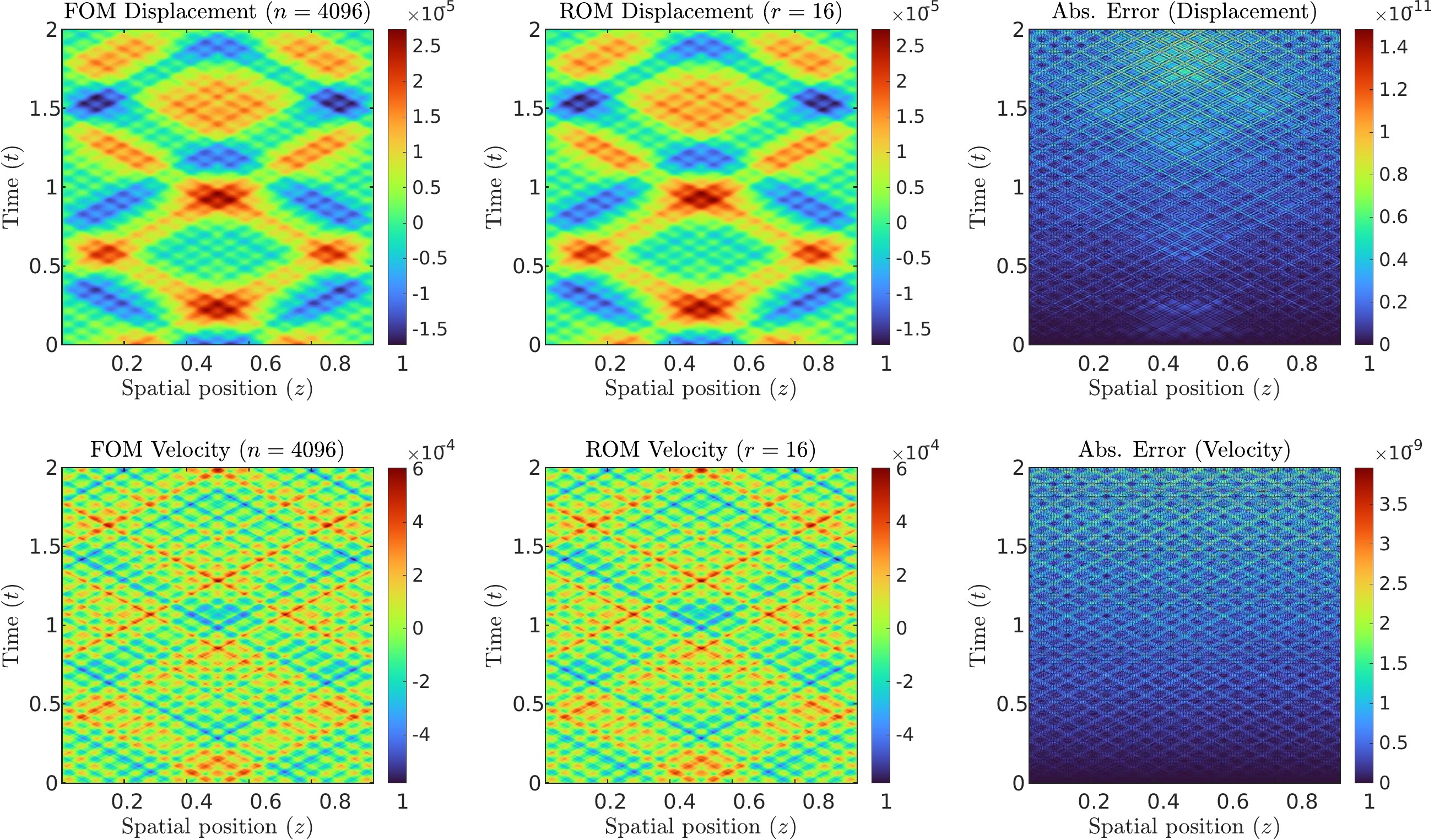}
    \caption{Performance evaluation of the framework demonstrating the high-fidelity tracking achieved by the quadratic ROM. Top: Comparison of the full-order and reduced-order trajectory dynamics (displacement) in the state space and absolute error made in (displacement) state reconstruction. Bottom: Comparison of the full-order and reduced-order trajectory dynamics (velocity) in the state space and absolute error made in (velocity) state reconstruction .}
    \label{fig:disp_velocity}
\end{figure}

We set the reduction order to $r=16$. The interpolation frequencies are chosen as complex conjugate pairs, spaced logarithmically along the imaginary axis between $10$ and $10^{2.5}$. This selection is motivated by the fact that the poles of the original system are located around this region. By spanning this frequency range, the reduced model accurately captures both the high-frequency components, which represent short-scale phenomena, and the low-frequency components, which dictate slow-scale dynamics. Following the procedure in Algorithm~\ref{alg:moment_matching_rom}, we compute the projection matrices $\bV_1$ and $\bV_2$ by solving the Sylvester equations in \eqref{eq:pi1_sylvester} and \eqref{eq:pi2_sylvester} and setting $\bW=\bV_1$.

\begin{figure}
    \centering
    \includegraphics[width=0.6\linewidth]{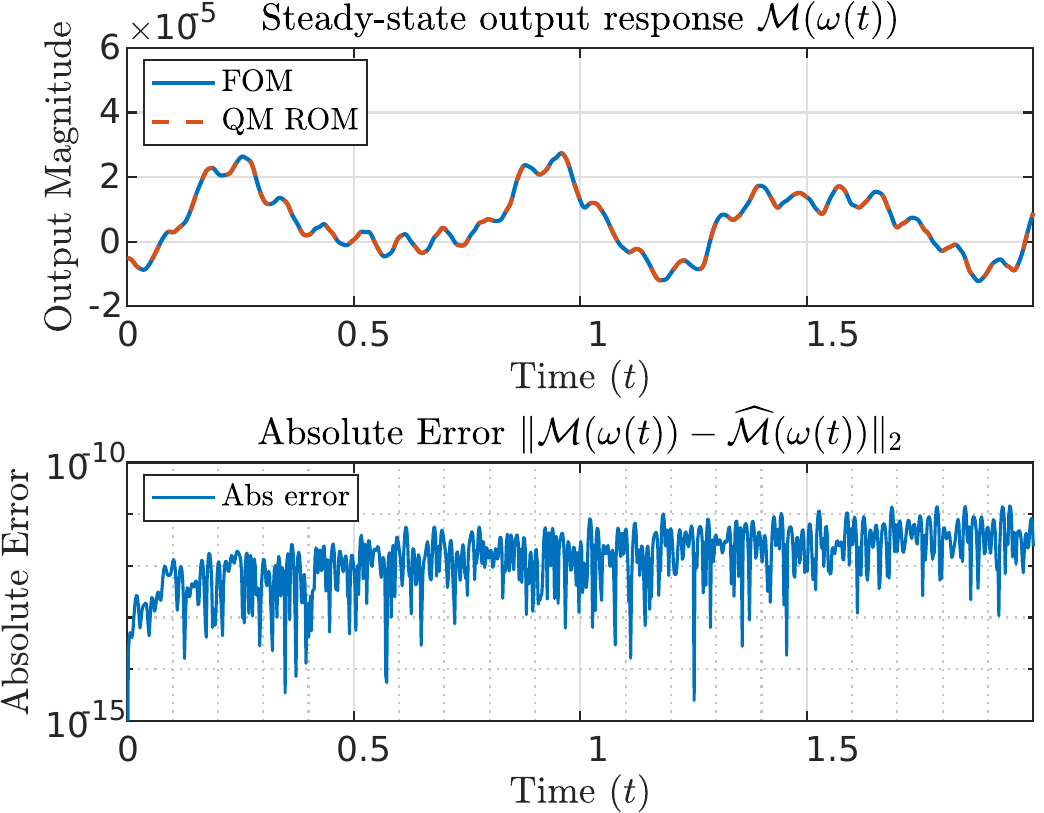}
    \caption{Numerical verification of the moment matching property in the one-dimesional wave equation}
    \label{fig:mom_match_wave}
\end{figure}

Just like in the advection equation example, we try to show that the quadratic ROM (\eqref{eq:rom_state}-- \eqref{eq:rom_group}) constructed from the algorithm indeed achieves moment matching at $(\bS,\bL_1,\bL_2)$. We generate an input signal generated by the signal generator with a random initial condition $\omega_0\in\mathbb R^r$ and plot the state evolution and output of both the full and reduced order systems, over $t\in[0,2]$, assuming that the two systems start on their respective center manifolds, i.e., \begin{align*}
    \bx(0)=\boldsymbol{\pi}(\w_0)=\bV_1\w_0 + \bV_2 \w_0^{(2)}\in\mathbb R^{4096}~~\mbox{and}~~\bx_r(0)=\boldsymbol{\pi}_r(\w_0)=\w_0\in\mathbb R^{16}.
\end{align*} To facilitate direct comparison between the low-dimensional state trajectory $\bx_r(t)\in\mathbb R^{16}$ and the full-order state trajectory $\bx(t)\in\mathbb R^{4096}$, the ROM state is lifted to the full-order space via the mapping\begin{align*}
    \boldsymbol{\Phi}(\mathbf{x}_r(t)) = \mathbf{V}_1 \mathbf{x}_r(t) + \mathbf{V}_2 (\mathbf{x}_r(t) \otimes \mathbf{x}_r(t)) \in \mathbb{R}^{4096}.
\end{align*} Figure~\ref{fig:disp_velocity} illustrates the state evolution of the FOM, state reconstruction of the QM-ROM and the absolute error in the state reconstruction of the QM-ROM. Note that the full state vector $\bx(t)$ is partitioned such that the first $n$ components correspond to the nodal displacements, while the remaining $n$ components represent the nodal velocities. The left panel displays the evolution of the state displacement and velocity versus the reduced-order approximation $\boldsymbol{\Phi}(\bx_r(t))$. It is evident that the quadratic ROM successfully captures the transient dynamics of the wave equation, remaining tightly bound to the full-order trajectory even as the wave propagates. The right panel quantifies this performance through the absolute error. The error remains small and bounded, confirming that the solution to the Sylvester equations \eqref{eq:pi1_sylvester} and \eqref{eq:pi2_sylvester} effectively identifies the nonlinear center manifold of the system. 

In addition to state-space reconstruction, Figure~\ref{fig:mom_match_wave} evaluates the steady-state output response $\mathcal{M}(\boldsymbol{\omega}(t))$ of the full-order system alongside the reduced-order output $\widehat{\mathcal{M}}(\boldsymbol{\omega}(t))$ obtained from the quadratic moment-matching ROM (QM ROM). As depicted in the top panel, the output trajectory predicted by the QM ROM is visually indistinguishable from the full-order response over the entire time interval $t \in [0, 2]$. The bottom panel provides the absolute error $\|\mathcal{M}(\boldsymbol{\omega}(t)) - \widehat{\mathcal{M}}(\boldsymbol{\omega}(t))\|_2$, which remains bounded between $10^{-12}$ and $10^{-11}$. This  serves as concrete numerical verification that the proposed ROM satisfies the moment-matching conditions in the steady-state output space for the one-dimensional wave equation. Table~\ref{tab:runtime_wave} reports the empirical online runtimes comparing the QM-ROM against the FOM, with the QM-ROM offering over a $330\times$ online speedup.

\subsection{Computational profiling and online performance}
\label{sec:computational_profiling}
In this section, we address the computational scaling of the framework with system dimension $n$. We use the term ``offline time" to refer to the time taken to compute the reduced system and the term ``online time" to mean the time taken to simulate the system once it is formed. We show that our method not only has a competitive online time scaling but also quick offline time. A potential computational bottleneck in quadratic manifold-based ROMs stems from the presence of the state-dependent reduced mass matrix, given in \eqref{eq:rom_group} as
\begin{equation}
    \mathbf{M}_r(\mathbf{x}_r) = \mathbf{I}_r + \mathbf{E}_r(\mathbf{I}_r \otimes \mathbf{x}_r + \mathbf{x}_r \otimes \mathbf{I}_r).
\end{equation}
Since $\mathbf{M}_r(\mathbf{x}_r)$ varies continuously with the reduced state $\mathbf{x}_r(t)$, a naive implementation would explicitly assemble the large, highly redundant Kronecker product matrix of size $r^2 \times r$ and solve a dense $r \times r$ linear system at every time step or stage of an ODE solver. Such an unoptimized approach would heavily penalize the execution runtime and jeopardize any computational gains. We provide two methods to effectively handle the simulation of this term without losing online speedups. 

We exploit the underlying algebraic structure of the Kronecker product, allowing us to evaluate the action of $\mathbf{M}_r(\mathbf{x}_r)$ without ever explicitly forming or storing the constituent high-dimensional matrices. Consider the term $\bM_r(\hat \bx)$  \begin{align*}
    \bM_r(\hat \bx):=(\bI_r+\bE_r(\hat \bx \otimes \bI_r + \bI_r \otimes \hat \bx))  \quad \text{with } \bE_r = \left[ \bE_r^{(1)} \quad \bE_r^{(2)} \quad \dots \quad \bE_r^{(r)} \right],
\end{align*} where $\bE_r\in\mathbb R^{r\times r^2}$ as defined in \eqref{eq:rom_group}, $\hat\bx\in\mathbb R^r$ is an arbitrary vector, and $\bE_r^{(i)}\in\mathbb R^{r\times r}$ represent the matrix blocks of $\bE_r$. Then using the properties of the Kronecker product, we have that \begin{align}
    \bM_r(\hat \bx) = \sum_{i=1}^r \bx_i \bE_r^{(i)} + \left[\bE_r^{(1)}\hat\bx \quad \bE_r^{(2)}\hat\bx \quad \dots \quad \bE_r^{(r)}\hat\bx \right].\label{eq:fast_kron}
\end{align} As we see in \eqref{eq:fast_kron}, the computation of $\bM_r(\hat\bx)$ can be done in $\mathcal O(r^3)$ flops where $r\ll n$, e.g., $r\approx 20$. Direct LU or Cholesky factorization of this assembled matrix also scales as $\mathcal{O}(r^3)$, thereby ensuring the online time complexity is cubic in $r$ and significantly cheaper than solving the full order system. The implementation details can be found in the repository~\cite{padhi_code}.

As an alternative, enforcing the condition $\bW^\top\bV_2=0$ (a standard practice in quadratic manifold approaches) results in $\bE_r=0$, simplifying the reduced mass matrix to the identity matrix. While our framework retains the generality of a state-dependent mass matrix to ensure strict moment matching, this choice of $\bW$ offers a viable path for scenarios where online efficiency is the paramount priority. 

Table~\ref{tab:runtime_comparison_combined} summarizes the online simulation runtimes for the full-order models and the proposed QM ROMs across both benchmark problems. The QM ROM achieves substantial online speedups, reducing execution times from hundreds of seconds down to under one second and thus yielding an approximate $388\times$ speedup for the 1D advection equation and a $330\times$ speedup for the 1D damped wave equation. This reduction in computational cost stems directly from confining the online time integration to the low-dimensional state space $\mathbb{R}^r$, effectively decoupling the online solver cost from the large full-order dimension $n$. 

\begin{table}
\centering
\caption{Computational wall-clock online runtime comparison between the full-order model (FOM) and proposed QM ROM across benchmark problems.}
\label{tab:runtime_comparison_combined}
\begin{subtable}[b]{0.48\linewidth}
    \centering
    \caption{1D Advection Equation}
    \label{tab:runtime_advection}
    \begin{tabular}{lc}
    \toprule
    \textbf{Method}  &\textbf{Online Time (s)} \\
    \midrule
    FOM ($n=1024$)  &378.4343 \\
    QM ROM ($r=20$) & \textbf{0.9752} \\
    \bottomrule
    \end{tabular}
\end{subtable}
\hfill
\begin{subtable}[b]{0.48\linewidth}
    \centering
    \caption{1D Damped Wave Equation}
    \label{tab:runtime_wave}
    \begin{tabular}{lc}
    \toprule
    \textbf{Method} & \textbf{Online Time (s)} \\
    \midrule
    FOM ($n=4096$) & 167.8696 \\
    QM ROM ($r=16$)& \textbf{0.5042} \\
    \bottomrule
    \end{tabular}
\end{subtable}
\end{table}

\subsection{Theoretical and numerical comparisons with linear methods}
\label{subsec:comparison_irka}

To contextualize the performance benefits of the proposed Quadratic Manifold (QM) framework, we evaluate its online time complexity and empirical performance against classical linear projection-based reduced-order models (ROMs), such as One-Sided Rational interpolation (OSR) \cite{morAntBG20}.

For transport-dominated or wave-propagation problems, e.g., the 1D advection equation in Section~\ref{sec:adv_eqn}, linear projection methods suffer from a fundamental barrier due to the slow decay of the Kolmogorov $n$-width (the Hankel singular values for the LTI systems). Thus, to reach an acceptable error tolerance, a linear subspace requires a large subspace dimension. As discussed in Section~\ref{subsec:diagonal_gen}, the nonlinear moment $\mathcal{M}(\omega)$ of the system is given by
\begin{equation}\label{eq:moment_numerics}
    \mathcal{M}(\omega) = \sum_{i=1}^{r} \mathbf{H}(s_i)\mathbf{\ell}_i \omega_i + \sum_{i=1}^{r}\sum_{j=1}^{r} \mathbf{H}(s_i + s_j)\mathbf{b}_{ij} \omega_i \omega_j.
\end{equation}
To match these dynamics, a linear ROM must interpolate the transfer function at $r + (r^2+r)/2$ unique points, increasing the required linear subspace dimension to $r_{\text{linear}} = \mathcal{O}(r^2)$. The online time-complexity for linear methods scales as $\mathcal O(r_{\text{linear}}^2)$ since their reduced matrices can be pre-computed and there is no time-dependent mass-matrix term. However, due to larger state dimension, the execution cost of an equivalent linear ROM scales as $\mathcal{O}(r_{\text{linear}}^2) = \mathcal{O}(r^4)$. In contrast, the QM framework confines the online ODE solver strictly to a low-dimensional state vector $\mathbf{x}_r(t) \in \mathbb{R}^{r}$. By exploiting tensor-structured evaluations in \eqref{eq:fast_kron}, the QM ROM achieves an online time complexity of $\mathcal{O}(r^3)$.

\begin{figure}
    \centering
    \includegraphics[width=0.8\linewidth]{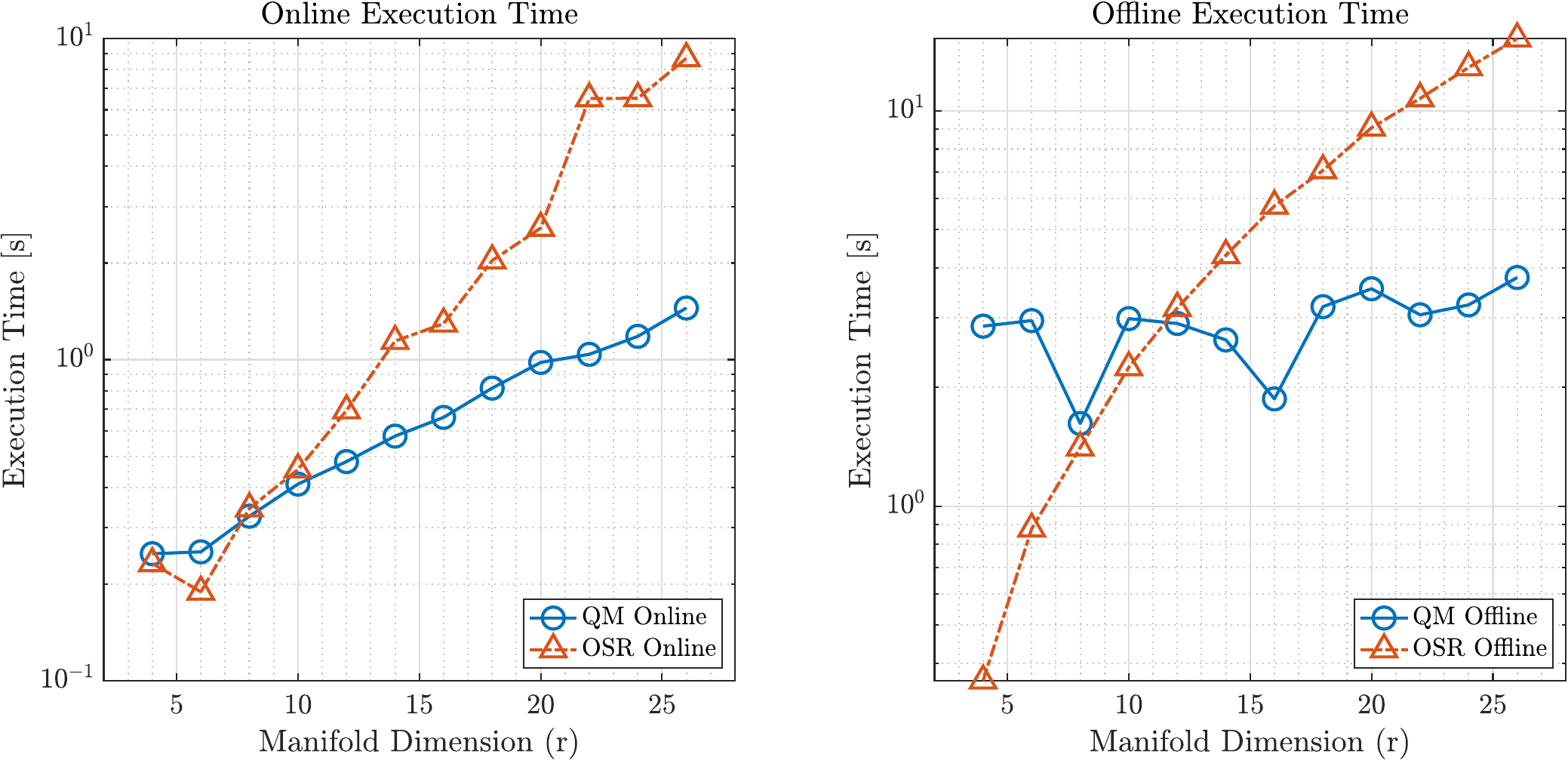}
    \caption{Comparative analysis of the offline and online runtime as a function of the reduced manifold dimension $r$ for One-Sided Rational Interpolation (OSR) and the proposed system-theoretic Quadratic Manifold (QM) framework.}
    \label{fig:comparison_r}
\end{figure}

To empirically validate these online runtime properties, we benchmark both frameworks on the 1D linear advection problem ($n = 1024$) across manifold dimensions $r \in [4, 26]$. We ensure a fair comparison by comparing a quadratic ROM of dimension $r$ with a linear ROM that interpolates the transfer function $\bH(s)$ at the respective $r+(r^2+r)/2$ points. For instance, a QM model with $r = 20$ is evaluated against a linear OSR model interpolating the transfer function at the corresponding 230 points. The OSR implementation constructs a linear projection matrix followed by orthogonalization for numerical stability, selecting only directions corresponding to non-zero singular values in rank-deficient cases. This results in an added offline cost when $r_\text{linear}$ is large. Empirical online runtimes are measured using the CPU time required by the ODE solver to integrate the ROM solution over $t \in [0, 1]$, as shown in Figure~\ref{fig:comparison_r}.

In the online phase (Figure~\ref{fig:comparison_r}, left), the linear OSR model exhibits rapid (theoretically $\mathcal{O}(r^4)$) runtime growth. Conversely, the QM ROM maintains a milder $\mathcal{O}(r^3)$ growth, delivering up to an order-of-magnitude wall-clock speedup at higher dimensions $r$. In the offline phase (Figure~\ref{fig:comparison_r}, right), QM remains highly scalable by solving decoupled Sylvester equations \eqref{eq:pi1_sylvester}--\eqref{eq:pi2_sylvester}. We note that for some choices of signal generator, $\bV_1$ and $\bV_2$ can be computed using their closed-form solutions as shown in Section~\ref{subsec:diagonal_gen}. Because both QM and OSR theoretically guarantee exact interpolation of the targeted nonlinear moments (as per Theorem~\ref{thm:moment_matching} and \eqref{eq:moment_numerics}), the QM framework achieves superior online computational efficiency without compromising on accuracy in this example.

\subsection{Comparison with existing quadratic manifold methods}
\label{sec:comparison_greedy}

To contextualize the efficacy and system-theoretic guarantees of the proposed interpolatory quadratic manifold approach, we benchmark it against state-of-the-art data-driven quadratic manifold techniques. In current literature, non-linear reduced-order models (ROMs) defined on quadratic manifolds are frequently constructed by optimizing for the linear and quadratic projection matrices, $\bV_1$ and $\bV_2$, using state snapshot trajectories. As a benchmark, we consider the data-driven greedy snapshot optimization framework introduced in~\cite{schwerdtner2024greedy}. To ensure a consistent baseline within an intrusive model reduction setting, both the greedy method and our proposed interpolatory formulation are provided full access to the full-order model (FOM) system operators $\bA, \bB,$ and $\bC$.

We evaluate both methods on the 1D advection equation benchmark detailed in Section~\ref{sec:adv_eqn}. A training dataset of $1{,}000$ state snapshots is collected over the initial interval $t \in [0, 1]$ by driving the FOM with the signal generator
\begin{align*}
    \begin{aligned}
    \bS &= \operatorname{blkdiag}\left(\left[\begin{array}{cc}
       0&\mu_1  \\
     \overline{\mu}_1 & 0
    \end{array}\right],  \dots, \left[\begin{array}{cc}
       0& \mu_{r/2}   \\
      \overline{\mu}_{r/2}& 0  
    \end{array}\right]\right)\in\mathbb R^{r\times r},
\end{aligned}
\end{align*}
which results in inputs of the form \begin{align*}
    \bu(t)= \left(\sum_{i=1}^r c_{i,1}\sin(\mu_i t) + c_{i,2}\cos(\mu_i t)\right) + \left(\sum_{i,j=1}^r c_{ij,1} \sin((\mu_i+\mu_j)t) + c_{ij,2} \cos((\mu_i+\mu_j)t)\right),
\end{align*}  where $c_{i,1},c_{i,2},c_{ij,1}, c_{ij,2}$ are constants dependent on $\bL_1,\bL_2,\w_0$ for $1\leq i,j,\leq r$ as presented in \eqref{eq:sin_inputs}. Following the greedy optimization algorithm from~\cite{schwerdtner2024greedy} with our implementation provided in~\cite{padhi_code}, projection matrices $\bV_{g,1}$ and $\bV_{g,2}$ are computed to minimize empirical state trajectory reconstruction error over the $t \in [0, 1]$ training window. Model accuracy is measured across the same training snapshot dataset. The test basis (left projection matrix) is set to $\bW_{g} = \bV_{g,1}$, yielding the greedy quadratic ROM (gQM-ROM) of dimension $r=20$, which has the form \eqref{eq:rom_state}--\eqref{eq:rom_group} with an identity mass-matrix term since $\bW_g^\top\bV_{g,2}=0$. The singular value decay of the snapshot data used to construct the greedy ROM projection matrices is shown in bottom left panel of Figure~\ref{fig:greedy_method}. On the other hand, Algorithm~\ref{alg:moment_matching_rom} is used to construct the proposed interpolatory QM-ROM ($r=20$).

\begin{figure}
    \centering
    \includegraphics[width=1\linewidth]{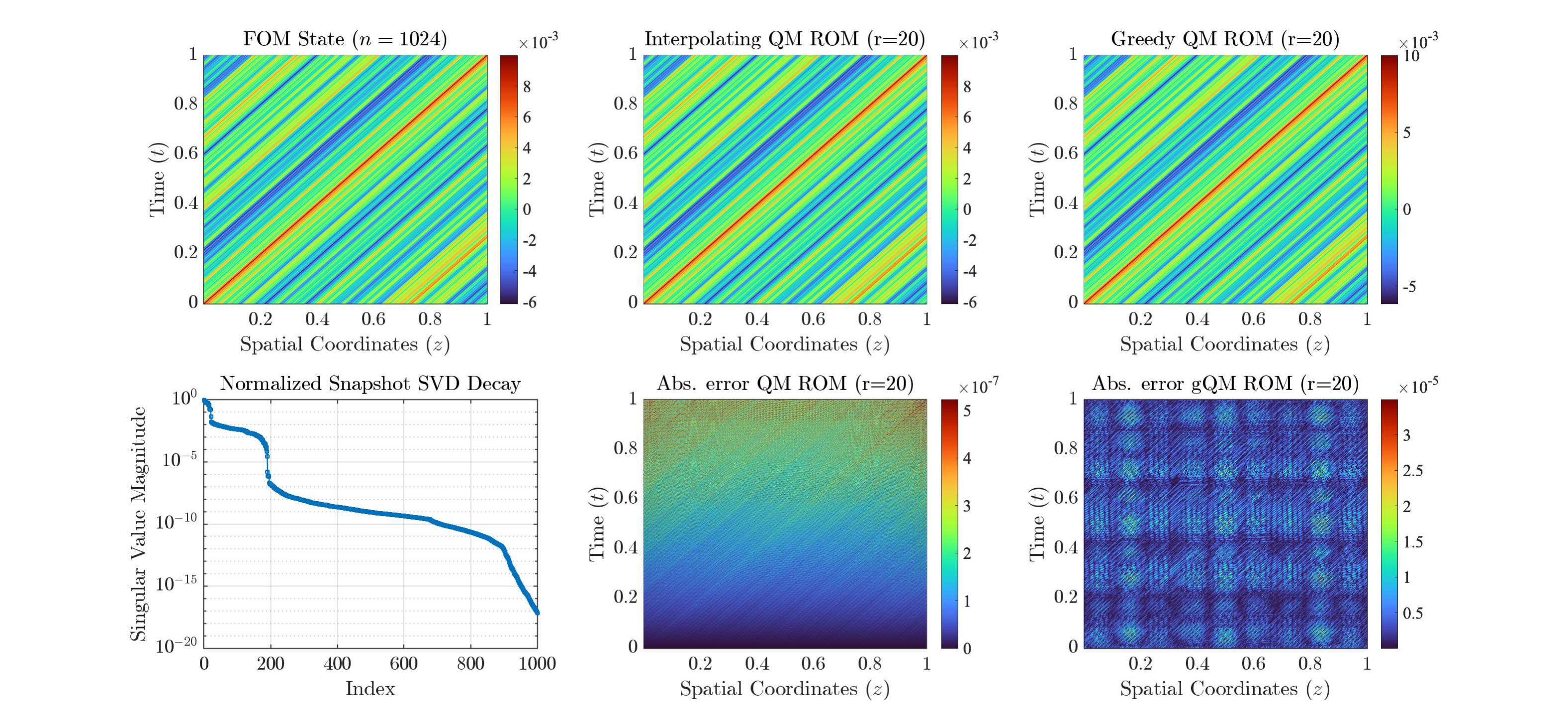}
    \caption{Spatio-temporal state reconstruction (top panel) and absolute error comparison (bottom panel) across the training  interval $t \in [0, 1]$  for the FOM, proposed Interpolatory QM-ROM, and Greedy QM-ROM (gQM-ROM). }
    \label{fig:greedy_method}
\end{figure}

\begin{table}
\centering
\caption{Computational wall-clock time and state reconstruction error comparison between the Full-Order Model (FOM), the proposed interpolatory QM-ROM, and the greedy data-driven QM ROM (gQM-ROM) at reduced order $r = 20$.}
\label{tab:greedy_timing_comparison}
\begin{tabular}{lccc}
\toprule
\textbf{Method} & \textbf{Offline Time (s)} & \textbf{Online Time (s)} & \textbf{Absolute Error} \\
\midrule
Full order model (Full) & --- & 147.0238 & ---  \\
Proposed interpolatory (QM-ROM) & \textbf{3.0298} & \textbf{0.7266 } & $\mathbf{6.196 \times 10^{-5}}$ \\
Greedy QM-ROM (gQM-ROM) \cite{schwerdtner2024greedy} & 23.4082  & 0.4709 & $1.622 \times 10^{-3}$ \\
\bottomrule
\end{tabular}
\end{table}

Figure~\ref{fig:greedy_method} illustrates the spatio-temporal dynamics and corresponding absolute reconstruction errors across the simulation window $t \in [0, 1]$. While both reduced models successfully capture the wave propagation, for this example, the proposed interpolatory QM-ROM achieves more than two orders of magnitude higher accuracy ($\mathcal{O}(10^{-5})$ compared to $\mathcal{O}(10^{-3})$ for gQM-ROM). This might be due to the fact that the snapshot-based approaches are based on minimizing an empirical least-squares 
state-reconstruction fitting over a finite ensemble of snapshots. While the use of a quadratic manifold significantly relaxes the dependence on slow singular value decay relative to linear subspace methods, namely an order-$r$ quadratic manifold aims to capture information up to approximately the $(r + r^2)$-th singular direction rather than just the $r$-th, the fitting residual at order $r$ is still governed by how quickly the neglected singular values decay beyond that effective order, as shown in the bottom left panel of Figure~\ref{fig:greedy_method}. 

On the other hand, for LTI systems, the proposed framework is based on interpolating the steady-state trajectory for a special class of inputs (see Section~\ref{sec:signal_space}). Consequently, the reduced order $r$ in our setting parameterizes the state dimension of the target exogenous signal generator (and thus the richness of the interpolated input family), instead of truncation rank for snapshot data.

{Table~\ref{tab:greedy_timing_comparison} details the computational statistics and reconstruction errors for both the proposed interpolatory QM-ROM and the greedy quadratic ROM (gQM-ROM). While the gQM-ROM provides a robust framework that can be tailored via specific hyperparameter selections, the optimization-free nature of the proposed formulation (Algorithm~\ref{alg:moment_matching_rom}) serves as a complementary approach, naturally yielding consistently low offline construction times. During the online phase, both methods share an identical theoretical time complexity of $\mathcal{O}(r^3)$ due to their quadratic state operators. In practice, the gQM-ROM achieves excellent online execution speeds (as reflected in Table~\ref{tab:greedy_timing_comparison}) by enforcing an identity mass-matrix structure ($\bM_r(\bx_r)=\mathbf{0}$). Notably, this structural advantage can be integrated into the proposed interpolatory QM-ROM (by picking $\bW$ such that $\bW^\top\bV_2=0$) if comparable runtime efficiency is desired.}

Finally, we emphasize that many parameters appear  in  the optimization-based QM methods, such as  snapshot sample density, time horizons, and regularization parameters. Thus, hyperparameter tuning could potentially yield further improvements in the accuracy of gQM-ROM beyond what we obtained here. However, the primary objective of this comparison is not to conduct an exhaustive optimization benchmark, especially considering that gQM-ROM is applicable to general nonlinear dynamical systems rather than the LTI systems we consider here. Moreover, 
gQM-ROM and other snapshot based quadratic manifold approaches are not generally employed in a projection setting, but rather in a data-driven modeling framework. 
Our main goal here, instead, is to demonstrate that the proposed interpolatory QM-ROM provides a compelling, system-theoretic alternative to existing data-driven quadratic manifold approaches. Using the theory of nonlinear moment matching, our framework bypasses  snapshot collection and hyperparameter tuning, delivering certified system-theoretical guarantees alongside computational savings for this example. Extending our framework to a data-driven setting and general nonlinear dynamical systems remain an ongoing research direction.

\section{Conclusions and future directions} \label{sec:conclusion}
 
We introduced an optimization-free algebraic framework for constructing quadratic manifold-based reduced-order models (ROMs). By moving beyond the inherent limitations of traditional linear subspaces, our approach provides a principled, system-theoretic pathway to overcome the slow decay of Kolmogorov $n$-widths that typically affects advection- and transport-dominated problems.
We established rigorous theoretical guarantees for the proposed framework, proving that the constructed ROM successfully matches the nonlinear moments of the full-order model. Furthermore, we demonstrated that the exact center manifold mapping is preserved, thereby ensuring the asymptotic tracking of steady-state outputs under specific classes of inputs.
The theoretical results were validated through numerical experiments on a transport-dominated one-dimensional damped wave equation and one-dimensional advection equation. The numerical results show that the quadratic manifold ROM captures the complex transient dynamics with high fidelity. This reduction directly translates to substantial online and offline computational savings. A natural extension is to generalize this system-theoretic quadratic manifold framework to broader classes of nonlinear or parameterized partial differential equations. Additionally, exploring schemes to optimally select the signal generator parameters will further expand the practical utility of quadratic manifold-based model order reduction.

\printcredits

\section*{Acknowledgements}
We thank Prof. Benjamin Peherstorfer and Prof. Steffen Werner  for their feedback  on the manuscript. 

\section*{Data availability}
The \matlab code for reproducing the plots and comparisons in this paper are available on Github \cite{padhi_code}. 

\bibliographystyle{abbrv}

\bibliography{cas-refs}

\begin{thebibliography}{10}

\bibitem{amsallem2012nonlinear}
D.~Amsallem, M.~J. Zahr, and C.~Farhat.
\newblock Nonlinear model order reduction based on local reduced-order bases.
\newblock {\em International Journal for Numerical Methods in Engineering}, 92(10):891--916, 2012.

\bibitem{ACA05}
A.~C. Antoulas.
\newblock {\em Approximation of Large-Scale Dynamical Systems}.
\newblock Adv. Des. Control 6. Society for Industrial and Applied Mathematics, Philadelphia, PA, 2005.

\bibitem{morAntBG20}
A.~C. Antoulas, C.~A. Beattie, and S.~Gugercin.
\newblock {\em Interpolatory Methods for Model Reduction}.
\newblock Computational Science \& Engineering. Society for Industrial and Applied Mathematics, Philadelphia, PA, 2020.

\bibitem{astolfi2010}
A.~Astolfi.
\newblock Model reduction by moment matching for linear and nonlinear systems.
\newblock {\em IEEE Transactions on Automatic Control}, 55(10):2321--2336, 2010.

\bibitem{astolfi_10year_survey}
A.~Astolfi, G.~Scarciotti, J.~Simard, N.~Faedo, and J.~V. Ringwood.
\newblock Model reduction by moment matching: Beyond linearity a review of the last 10 years.
\newblock In {\em 2020 59th IEEE conference on decision and control (CDC)}, pages 1--16. IEEE, 2020.

\bibitem{bai2022model}
H.~Bai, T.~Mylvaganam, and G.~Scarciotti.
\newblock Model reduction for quadratic-bilinear systems using nonlinear moments.
\newblock In {\em 2022 European Control Conference (ECC)}, pages 1702--1707. IEEE, 2022.

\bibitem{bai2006projection}
Z.~Bai and D.~Skoogh.
\newblock A projection method for model reduction of bilinear dynamical systems.
\newblock {\em Linear algebra and its applications}, 415(2-3):406--425, 2006.

\bibitem{BARNETT_QM1}
J.~Barnett and C.~Farhat.
\newblock Quadratic approximation manifold for mitigating the {K}olmogorov barrier in nonlinear projection-based model order reduction.
\newblock {\em Journal of Computational Physics}, 464:111348, 2022.

\bibitem{benner2015two}
P.~Benner and T.~Breiten.
\newblock Two-sided projection methods for nonlinear model order reduction.
\newblock {\em SIAM Journal on Scientific Computing}, 37(2):B239--B260, 2015.

\bibitem{BenB17}
P.~Benner and T.~Breiten.
\newblock {\em Chapter 6: Model Order Reduction Based on System Balancing}, pages 261--295.
\newblock SIAM, 2017.

\bibitem{benner2023quadratic}
P.~Benner, P.~Goyal, J.~Heiland, and I.~Pontes~Duff.
\newblock A quadratic decoder approach to nonintrusive reduced-order modeling of nonlinear dynamical systems.
\newblock {\em PAMM}, 23(1):e202200049, 2023.

\bibitem{benner2024structured}
P.~Benner, S.~Gugercin, and S.~W. Werner.
\newblock Structured interpolation for multivariate transfer functions of quadratic-bilinear systems.
\newblock {\em Advances in Computational Mathematics}, 50(2):18, 2024.

\bibitem{beyn2004freezing}
W.-J. Beyn and V.~Th{\"u}mmler.
\newblock Freezing solutions of equivariant evolution equations.
\newblock {\em SIAM Journal on Applied Dynamical Systems}, 3(2):85--116, 2004.

\bibitem{black2020projection}
F.~Black, P.~Schulze, and B.~Unger.
\newblock Projection-based model reduction with dynamically transformed modes.
\newblock {\em ESAIM: Mathematical Modelling and Numerical Analysis}, 54(6):2011--2043, 2020.

\bibitem{breiten2012interpolation}
T.~Breiten and P.~Benner.
\newblock Interpolation-based $\mathcal{H}_2$-model reduction of bilinear control system.
\newblock {\em SIAM J. Matrix Anal. Appl}, 33(3):859--885, 2012.

\bibitem{breiten2010krylov}
T.~Breiten and T.~Damm.
\newblock Krylov subspace methods for model order reduction of bilinear control systems.
\newblock {\em Systems \& Control Letters}, 59(8):443--450, 2010.

\bibitem{brewer1978kronecker}
J.~Brewer.
\newblock Kronecker products and matrix calculus in system theory.
\newblock {\em IEEE Transactions on Circuits and Systems}, 25:772--781, 1978.

\bibitem{buchfink2023symplectic}
P.~Buchfink, S.~Glas, and B.~Haasdonk.
\newblock Symplectic model reduction of {H}amiltonian systems on nonlinear manifolds and approximation with weakly symplectic autoencoder.
\newblock {\em SIAM Journal on Scientific Computing}, 45(2):A289--A311, 2023.

\bibitem{buchfink2024approximation}
P.~Buchfink, S.~Glas, and B.~Haasdonk.
\newblock Approximation bounds for model reduction on polynomially mapped manifolds.
\newblock {\em Comptes Rendus. Math{\'e}matique}, 362(G13):1881--1891, 2024.

\bibitem{QM_framework}
P.~Buchfink, S.~Glas, B.~Haasdonk, and B.~Unger.
\newblock Model reduction on manifolds: A differential geometric framework.
\newblock {\em Physica D: Nonlinear Phenomena}, 468:134299, 2024.

\bibitem{burela2023parametric}
S.~Burela, P.~Krah, and J.~Reiss.
\newblock Parametric model order reduction for a wildland fire model via the shifted {POD} based deep learning method.
\newblock {\em arXiv preprint arXiv:2304.14872}, 2023.

\bibitem{carr2012applications}
J.~Carr.
\newblock {\em Applications of centre manifold theory}.
\newblock Springer Science \& Business Media, 2012.

\bibitem{daniel2022physics}
T.~Daniel, F.~Casenave, N.~Akkari, A.~Ketata, and D.~Ryckelynck.
\newblock Physics-informed cluster analysis and a priori efficiency criterion for the construction of local reduced-order bases.
\newblock {\em Journal of Computational Physics}, 458:111120, 2022.

\bibitem{flagg2015multipoint}
G.~Flagg and S.~Gugercin.
\newblock Multipoint volterra series interpolation and {$\mathcal{H}_2$} optimal model reduction of bilinear systems.
\newblock {\em SIAM Journal on Matrix Analysis and Applications}, 36(2):549--579, 2015.

\bibitem{fresca2021comprehensive}
S.~Fresca, L.~Dede’, and A.~Manzoni.
\newblock A comprehensive deep learning-based approach to reduced order modeling of nonlinear time-dependent parametrized {PDE}s.
\newblock {\em Journal of Scientific Computing}, 87(2):61, 2021.

\bibitem{fresca2022pod}
S.~Fresca and A.~Manzoni.
\newblock {POD-DL-ROM}: Enhancing deep learning-based reduced order models for nonlinear parametrized {PDEs} by proper orthogonal decomposition.
\newblock {\em Computer Methods in Applied Mechanics and Engineering}, 388:114181, 2022.

\bibitem{gallivan2004sylvester}
K.~Gallivan, A.~Vandendorpe, and P.~Van~Dooren.
\newblock Sylvester equations and projection-based model reduction.
\newblock {\em Journal of Computational and Applied Mathematics}, 162(1):213--229, 2004.

\bibitem{geelen2023learning}
R.~Geelen, L.~Balzano, and K.~Willcox.
\newblock Learning latent representations in high-dimensional state spaces using polynomial manifold constructions.
\newblock In {\em 2023 62nd IEEE Conference on Decision and Control (CDC)}, pages 4960--4965. IEEE, 2023.

\bibitem{geelen2024learning}
R.~Geelen, L.~Balzano, S.~Wright, and K.~Willcox.
\newblock Learning physics-based reduced-order models from data using nonlinear manifolds.
\newblock {\em Chaos: An Interdisciplinary Journal of Nonlinear Science}, 34(3), 2024.

\bibitem{geelen2022localized}
R.~Geelen and K.~Willcox.
\newblock Localized non-intrusive reduced-order modelling in the operator inference framework.
\newblock {\em Philosophical Transactions of the Royal Society A: Mathematical, Physical and Engineering Sciences}, 380(2229), 2022.

\bibitem{geelen2023operator}
R.~Geelen, S.~Wright, and K.~Willcox.
\newblock Operator inference for non-intrusive model reduction with quadratic manifolds.
\newblock {\em Computer Methods in Applied Mechanics and Engineering}, 403:115717, 2023.

\bibitem{glas2026structure}
S.~Glas and H.~Mu.
\newblock Structure-preserving model reduction on manifolds of port-{H}amiltonian systems.
\newblock {\em arXiv preprint arXiv:2603.08656}, 2026.

\bibitem{gosea2018data}
I.~V. Gosea and A.~C. Antoulas.
\newblock Data-driven model order reduction of quadratic-bilinear systems.
\newblock {\em Numerical Linear Algebra with Applications}, 25(6):e2200, 2018.

\bibitem{graham2018kronecker}
A.~Graham.
\newblock {\em Kronecker products and matrix calculus with applications}.
\newblock Courier Dover Publications, 2018.

\bibitem{kol_wave_decay}
C.~Greif and K.~Urban.
\newblock Decay of the {K}olmogorov {$n$}-width for wave problems.
\newblock {\em Applied Mathematics Letters}, 96:216--222, 2019.

\bibitem{gu2011qlmor}
C.~Gu.
\newblock Qlmor: A projection-based nonlinear model order reduction approach using quadratic-linear representation of nonlinear systems.
\newblock {\em IEEE Transactions on Computer-Aided Design of Integrated Circuits and Systems}, 30(9):1307--1320, 2011.

\bibitem{hesthaven2026nonlinear}
J.~S. Hesthaven, B.~Peherstorfer, and B.~Unger.
\newblock Nonlinear model reduction for transport-dominated problems.
\newblock {\em Acta Numerica}, 35:173--272, 2026.

\bibitem{huang2004nonlinear}
J.~Huang.
\newblock {\em Nonlinear output regulation: theory and applications}.
\newblock SIAM, 2004.

\bibitem{kadeethum2022non}
T.~Kadeethum, F.~Ballarin, Y.~Choi, D.~O’Malley, H.~Yoon, and N.~Bouklas.
\newblock Non-intrusive reduced order modeling of natural convection in porous media using convolutional autoencoders: comparison with linear subspace techniques.
\newblock {\em Advances in Water Resources}, 160:104098, 2022.

\bibitem{kim2022fast}
Y.~Kim, Y.~Choi, D.~Widemann, and T.~Zohdi.
\newblock A fast and accurate physics-informed neural network reduced order model with shallow masked autoencoder.
\newblock {\em Journal of Computational Physics}, 451:110841, 2022.

\bibitem{klein2025entropy}
R.~Klein, B.~Sanderse, P.~Costa, R.~Pecnik, and R.~Henkes.
\newblock Entropy-stable model reduction of one-dimensional hyperbolic systems using rational quadratic manifolds.
\newblock {\em Journal of Computational Physics}, 528:113817, 2025.

\bibitem{krah2025robust}
P.~Krah, A.~Marmin, B.~Zorawski, J.~Reiss, and K.~Schneider.
\newblock A robust shifted proper orthogonal decomposition: Proximal methods for decomposing flows with multiple transports.
\newblock {\em SIAM Journal on Scientific Computing}, 47(2):A633--A656, 2025.

\bibitem{moreschini2025moment}
A.~Moreschini, M.~Scandella, A.~Astolfi, and T.~Parisini.
\newblock Moment matching by kernel-based learning.
\newblock {\em IEEE Transactions on Automatic Control}, 2025.

\bibitem{ohlberger2013nonlinear}
M.~Ohlberger and S.~Rave.
\newblock Nonlinear reduced basis approximation of parameterized evolution equations via the method of freezing.
\newblock {\em Comptes Rendus Mathematique}, 351(23-24):901--906, 2013.

\bibitem{otto2023learning}
S.~E. Otto, G.~R. Macchio, and C.~W. Rowley.
\newblock Learning nonlinear projections for reduced-order modeling of dynamical systems using constrained autoencoders.
\newblock {\em Chaos: An Interdisciplinary Journal of Nonlinear Science}, 33(11), 2023.

\bibitem{padhi_code}
R.~Padhi.
\newblock {Beyond linear subspaces: Nonlinear moment matching meets quadratic manifolds}.
\newblock \url{https://github.com/Rewtus/Moment-matching-quadratic-manifolds}, 2026.

\bibitem{papapicco2022neural}
D.~Papapicco, N.~Demo, M.~Girfoglio, G.~Stabile, and G.~Rozza.
\newblock The neural network shifted-proper orthogonal decomposition: a machine learning approach for non-linear reduction of hyperbolic equations.
\newblock {\em Computer Methods in Applied Mechanics and Engineering}, 392:114687, 2022.

\bibitem{paxton2026fast}
G.~Paxton, S.~Cheon, R.~Geelen, and S.~A. McQuarrie.
\newblock Fast quadratic manifold learning for nonlinear dimensionality reduction in large-scale systems using {R}iemannian optimization.
\newblock {\em arXiv preprint arXiv:2605.26039}, 2026.

\bibitem{peherstorfer2022breaking}
B.~Peherstorfer.
\newblock Breaking the {K}olmogorov barrier with nonlinear model reduction.
\newblock {\em Notices of the American Mathematical Society}, 69(5):725--733, 2022.

\bibitem{pinkus2012n}
A.~Pinkus.
\newblock {\em {$N$}-width in Approximation Theory}.
\newblock Springer Science \& Business Media, 2012.

\bibitem{reiss2021optimization}
J.~Reiss.
\newblock Optimization-based modal decomposition for systems with multiple transports.
\newblock {\em SIAM Journal on Scientific Computing}, 43(3):A2079--A2101, 2021.

\bibitem{reiss2018shifted}
J.~Reiss, P.~Schulze, J.~Sesterhenn, and V.~Mehrmann.
\newblock The shifted proper orthogonal decomposition: A mode decomposition for multiple transport phenomena.
\newblock {\em SIAM Journal on Scientific Computing}, 40(3):A1322--A1344, 2018.

\bibitem{rewienski2003trajectory}
M.~Rewienski and J.~White.
\newblock A trajectory piecewise-linear approach to model order reduction and fast simulation of nonlinear circuits and micromachined devices.
\newblock {\em IEEE Transactions on computer-aided design of integrated circuits and systems}, 22(2):155--170, 2003.

\bibitem{rowley2003reduction}
C.~W. Rowley, I.~G. Kevrekidis, J.~E. Marsden, and K.~Lust.
\newblock Reduction and reconstruction for self-similar dynamical systems.
\newblock {\em Nonlinearity}, 16(4):1257--1275, 2003.

\bibitem{rowley2000reconstruction}
C.~W. Rowley and J.~E. Marsden.
\newblock Reconstruction equations and the {K}arhunen--{L}o{\`e}ve expansion for systems with symmetry.
\newblock {\em Physica D: Nonlinear Phenomena}, 142(1-2):1--19, 2000.

\bibitem{rugh1981nonlinear}
W.~J. Rugh.
\newblock {\em Nonlinear system theory}.
\newblock Johns Hopkins University Press Baltimore, 1981.

\bibitem{rutzmoser2017generalization}
J.~B. Rutzmoser, D.~J. Rixen, P.~Tiso, and S.~Jain.
\newblock Generalization of quadratic manifolds for reduced order modeling of nonlinear structural dynamics.
\newblock {\em Computers \& Structures}, 192:196--209, 2017.

\bibitem{scarciotti2017data}
G.~Scarciotti and A.~Astolfi.
\newblock Data-driven model reduction by moment matching for linear and nonlinear systems.
\newblock {\em Automatica}, 79:340--351, 2017.

\bibitem{scarciotti2017nonlinear}
G.~Scarciotti and A.~Astolfi.
\newblock Nonlinear model reduction by moment matching.
\newblock {\em Foundations and Trends in System and Control}, 4(3-4):224--409, 2017.

\bibitem{scarciotti2024interconnection}
G.~Scarciotti and A.~Astolfi.
\newblock Interconnection-based model order reduction-a survey.
\newblock {\em European Journal of Control}, 75:100929, 2024.

\bibitem{schwerdtner2025empirical}
P.~Schwerdtner, S.~Gugercin, and B.~Peherstorfer.
\newblock Empirical sparse regression on quadratic manifolds.
\newblock {\em SIAM Journal on Scientific Computing}, 47(6):A3085--A3107, 2025.

\bibitem{schwerdtner2024online}
P.~Schwerdtner, P.~Mohan, A.~Pachalieva, J.~Bessac, D.~O'Malley, and B.~Peherstorfer.
\newblock Online learning of quadratic manifolds from streaming data for nonlinear dimensionality reduction and nonlinear model reduction.
\newblock {\em arXiv preprint arXiv:2409.02703}, 2024.

\bibitem{schwerdtner2025online}
P.~Schwerdtner, P.~Mohan, A.~Pachalieva, J.~Bessac, D.~O’Malley, and B.~Peherstorfer.
\newblock Online learning of quadratic manifolds from streaming data for nonlinear dimensionality reduction and nonlinear model reduction.
\newblock {\em Proceedings of the Royal Society A: Mathematical, Physical and Engineering Sciences}, 481(2314), 2025.

\bibitem{schwerdtner2024greedy}
P.~Schwerdtner and B.~Peherstorfer.
\newblock Greedy construction of quadratic manifolds for nonlinear dimensionality reduction and nonlinear model reduction.
\newblock {\em SIAM Journal on Mathematics of Data Science}, 8(3):793--819, 2026.

\bibitem{sharma2023symplectic}
H.~Sharma, H.~Mu, P.~Buchfink, R.~Geelen, S.~Glas, and B.~Kramer.
\newblock Symplectic model reduction of {H}amiltonian systems using data-driven quadratic manifolds.
\newblock {\em Computer Methods in Applied Mechanics and Engineering}, 417:116402, 2023.

\bibitem{simard2024parameterization}
J.~D. Simard, A.~Moreschini, and A.~Astolfi.
\newblock Parameterization of all differential-algebraic moment matching interpolants.
\newblock {\em IEEE Transactions on Automatic Control}, 70(3):1875--1882, 2024.

\bibitem{unger2019kolmogorov}
B.~Unger and S.~Gugercin.
\newblock Kolmogorov {$n$}-widths for linear dynamical systems.
\newblock {\em Advances in Computational Mathematics}, 45(5):2273--2286, 2019.

\bibitem{volkwein2011pod}
S.~Volkwein.
\newblock Model reduction using proper orthogonal decomposition.
\newblock Lecture notes, Institute of Mathematics and Scientific Computing, University of Graz, 2011.

\bibitem{werner2021structure}
S.~W. Werner.
\newblock {\em Structure-preserving model reduction for mechanical systems}.
\newblock PhD thesis, Otto-von-Guericke-Universität Magdeburg, Fakultät für Mathematik, 2021.

\end{thebibliography}

\end{document}